\documentclass{amsart}
\usepackage{mathrsfs}
\usepackage{amsfonts}
\usepackage{amsmath}
\usepackage{amssymb}
\usepackage{bbm}
\usepackage{amsthm}
\usepackage{enumerate}
\usepackage{color}
\usepackage{tikz}

\usetikzlibrary{arrows,shapes,chains}
\usepackage{shuffle}

\numberwithin{equation}{section}

\allowdisplaybreaks
\theoremstyle{definition}
\newtheorem{theorem}{Theorem}[section]
\newtheorem{definition}[theorem]{Definition}

\newtheorem{lemma}[theorem]{Lemma}
\newtheorem{corollary}[theorem]{Corollary}
\newtheorem{question}[theorem]{Question}
\newtheorem{proposition}[theorem]{Proposition}
\newtheorem{example}[theorem]{Example}

\newtheorem{remark}[theorem]{Remark}

\newcommand{\AG}{\mathbb{A}}

\newcommand{\Ar}{\operatorname{Ar}}
\newcommand{\Com}{{\mathcal{C}om}}
\newcommand{\extn}{\mathbf{ex}}

\newcommand{\Ext}{\operatorname{Ext}}

\newcommand{\GKdim}{\operatorname{GKdim}}
\newcommand{\GK}{\operatorname{GKdim}}

\newcommand{\II}{\mathcal{I}}
\newcommand{\Irr}{\operatorname{Irr}}
\newcommand{\Lin}{{\mathcal{L}in}}
\newcommand{\Mas}{{\mathcal{M}as}}
\newcommand{\NGAss}{{\mathbb N}gr{\mathcal{A}ss}}
\newcommand{\Ope}{{\mathcal{O}pe}}
\newcommand{\Op}{{\mathcal{O}p}}
\newcommand{\PP}{\mathcal{P}}
\newcommand{\QQ}{\mathcal{Q}}

\newcommand{\sign}{\mathbf{sg}}
\newcommand{\sgn}{\operatorname{sgn}}

\newcommand{\SG}{\mathbb{S}}

\newcommand{\triv}{\mathbf{tr}}

\newcommand{\TX}{\mathcal{T}(\mathcal X)}

\newcommand{\VV}{\mathcal{V}}

\newcommand{\WW}{\mathcal{W}}
\newcommand{\wt}{\operatorname{wt}}
\newcommand{\XX}{\mathcal{X}}

\newcommand{\NN}{\mathbb{N}}

\usepackage[pdftex,linkcolor=blue,citecolor=blue,backref=page]{hyperref}

\title{Symmetric operads of GK-dimension one}
\author[Li]{Yu Li}
\address{School of Mathematics and Statistics,
Huizhou University, Huizhou, Guangdong 516007, China}
\email{liyu820615@126.com}

\author[Qi]{Zihao Qi}
\address{Department of Mathematics,
Fudan University, Shanghai 200433, China}
\email{qizihao@foxmail.com}

\author[Xu]{Yongjun Xu${}^*$}
\address{School of Mathematical Sciences, Qufu Normal University,
Qufu 273165, China}
\email{yjxu2002@163.com}

\author[Zhang]{James J. Zhang}
\address{Department of Mathematics, Box 354350,
University of Washington, Seattle, Washington 98195, USA}
\email{zhang@math.washington.edu}

\author[Zhang]{Zerui Zhang}
\address{School of Mathematical Sciences,
South China Normal University,
Guangzhou 51063, China}
\email{zeruizhang@scnu.edu.cn}

\author[Zhao]{Xiangui Zhao}
\address{School of Mathematics and Statistics,
Huizhou University, Huizhou, Guangdong 516007, China}
\email{zhaoxg@hzu.edu.cn}

\thanks{${}^*$ Corresponding author}



\begin{document}

\begin{abstract}
We prove that there is no finitely generated symmetric operad
of Gelfand-Kirillov dimension strictly between 1 and 2 that
answers an open question posted in 2020. We also classify finitely
generated prime symmetric operads of Gelfand-Kirillov dimension 1.
\end{abstract}

\maketitle

\setcounter{section}{-1} 
\section{Introduction}
\label{xxsec0}
The concept of an operad was introduced by Boardman and Vogt \cite{BV73} and May \cite{May72} in the early 1970s within the framework of homotopy theory. Since then, operads have become important objects of study across a wide spectrum of mathematical disciplines.
One of the significant developments in the theory of operads is Ginzburg and Kapranov's foundational work on Koszul duality \cite{GK94}.  Dotsenko and Khoroshkin \cite{DK10} established the theory of Gr\"obner bases for operads, yielding an effective algorithmic version of Hoffbeck's PBW criterion for Koszulness in (symmetric) quadratic operads. Subsequently, Khoroshkin and Piontkovski \cite{KP15} investigated generating functions for the dimensions of graded components $\PP(n)$   for arbitrary operads $\PP$ admitting finite Gr\"obner bases.
Dotsenko \cite{Dot19} described a forgetful functor from the category of differential graded operads to that of weight graded differential graded algebras, and constructed two notions of ``enveloping operads" for arbitrary weight graded differential graded algebras. These fundamental results  motivate us to further study algebraic properties of operads.  In the present article, we continue to study the connection between graded algebras and operads.

Operad theory has attracted considerable interest among algebraists in recent years. Wang and Zhou \cite{WZ24} examined minimal modules of Rota-Baxter algebras of arbitrary weight from an operadic perspective. Dotsenko and Umirbaev \cite{DU23} established an effective criterion for Nielsen-Schreier varieties by applying Gr\"obner basis theory for operads. Dotsenko and Tamaroff \cite{DT21} developed a categorical approach to Poincar\'{e}-Birkhoff-Witt  type theorems through the study of module actions on operads.
The Gelfand-Kirillov (GK) dimension of an operad was first defined by Khoroshkin and Piontkovski \cite{KP15}. Bao, Ye, and Zhang \cite{BYZ20} introduced truncation ideals for  unitary symmetric operads, employing them to classify operads of low Gelfand-Kirillov dimension. This line of investigation was continued by Bao, Fu, Ye, and Zhang \cite{BFYZ25}, who classified 2-unitary operads of Gelfand-Kirillov dimension 3. 
 In \cite[Theorem 1.1]{QXZZ20},  the authors proved a version of
Bergman's gap theorem for nonsymmetric operads: {\it there is no
finitely generated nonsymmetric operad with Gelfand-Kirillov
dimension strictly between 1 and 2}. 
Note that, for every real number $d>3$, there is a finitely
generated symmetric operad with Gelfand-Kirillov dimension $d$
\cite[Theorem 1.7(2)]{QXZZ20}. 
\cite[Question 1.8]{QXZZ20}
asks if the Bergman's gap theorem holds for symmetric operads.
One of the main results of this paper is to answer this question
affirmatively:

\begin{theorem} [Bergman's gap theorem]
\label{zzthm0.1}
The Gelfand-Kirillov dimension of a finitely generated symmetric
operad cannot be strictly between 1 and 2.
\end{theorem}

Bergman's gap theorem was originally
proved for associative algebras, namely, there is no associative
algebra with Gelfand-Kirillov dimension strictly between 1 and 2
\cite[Theorem 2.5]{KL00}. Bergman's gap theorem also holds for several other classes of algebras, for
example, Jordan algebras \cite{MZ96}, dialgebras \cite{zhang2019no}
and brace algebras \cite{LMZZ25}.
However, Bergman's gap theorem fails for some other types of algebras
such as Lie algebras \cite{Pet97} and Jordan superalgebras \cite{PS19}.

Bergman's result is fundamentally important for understanding
associative algebras of low Gelfand-Kirillov dimension.
Theorem \ref{zzthm0.1} leads us to further study symmetric operads
of low Gelfand-Kirillov dimension. As an application of Theorem~\ref{zzthm0.1}, this paper provides three
classification results of different classes of symmetric operads
of Gelfand-Kirillov dimension 1. By~\cite[Proposition 0.5]{BYZ20}, the only 2-unitary operad
of GK-dimension 1 is the operad endowing unital commutative algebras.
In the reduced case which we assume in this article, there are
more operads of GK-dimension 1.

Throughout let $\Bbbk$ be a base field
and $\overline{\Bbbk}$ be the algebraic closure of $\Bbbk$.  All algebraic objects
are over $\Bbbk$ unless otherwise stated. Let $\dim$ denote the $\Bbbk$-vector space dimension. The definition of
Gelfand-Kirillov dimension (or GK-dimension for short) of an
operad will be recalled in Definition \ref{zzdef1.2}.  
By an operad, we always mean a reduced (i.e.,
$\PP(0)=0$), locally finite (i.e., $\dim  \PP(n)<\infty$
for all $n$), symmetric operad unless otherwise stated.
Similarly,
by a graded algebra,
we always mean a locally finite graded algebra unless otherwise stated.

Let $\Com_{\Bbbk}$ (or $\Com$) be the operad that encodes nonunital commutative algebras
over $\Bbbk$. Now we recall another operad that encodes all
{\it skew-symmetric totally associative
ternary} algebras \cite[Section 2.1]{AM09}.

\begin{example} [{\cite{AM09}}]
\label{zzex0.2}
Let $\Ope_{\Bbbk}$ be the symmetric operad  defined by
\begin{enumerate}
\item[(i)]
$\Ope_{\Bbbk}(n)=\begin{cases} 0, & {\text{$n$ is even,}}\\
\Bbbk \mu_n, & {\text{$n$ is odd.}}
\end{cases}$
\item[(ii)]
$\mu_n \ast \sigma=\sgn(\sigma)\mu_n$ for all
$\sigma\in \SG_n$, where $\sgn(\sigma)$ is the sign of
$\sigma$,
\item[(iii)]
when both $n$ and $m$  are odd, $\mu_n\circ_i \mu_m=\mu_{n+m-1}$
for every $1\leq i\leq n$.
\end{enumerate}
Then $\Ope_{\Bbbk}$ is called the {\it odd power exterior operad}. When there is no confusion, we usually denote $\Ope_{\Bbbk}$ by $\Ope$. By \cite[Section 2.1]{AM09}, algebras over $\Ope$ are called \emph{skew-symmetric
totally associative ternary algebras}. 
\end{example}

 Now we are ready to state our first result on classification of {\it prime operads}  (Definition \ref{zzdef1.5}(1)) of
GK-dimension 1.

\begin{theorem} [${\text{Classification of prime operads of
GK-dimension 1}}$]
\label{zzthm0.3}
Let $\PP$ be a finitely generated prime operad over an  algebraically closed field~$\Bbbk$. If $\PP$ is of GK-dimension 1,
then $\PP$ is a suboperad of $\Com$ or $\Ope$.
\end{theorem}
An element $\mu$ in an operad $\PP$ is called
{\it central}, if  for every $\nu\in \PP$  the equation
\begin{equation*}
\mu\circ_i \nu =\nu\circ_j \mu
\end{equation*}
holds for all $1\leq i\leq \Ar(\mu)$ and $1\leq j\leq \Ar(\nu)$. 
Theorem \ref{zzthm0.3} shows that every element in a finitely generated prime operad of GK-dimension 1 over an  algebraically closed field~$\Bbbk$ is central.
However, if we remove the hypotheses on ``finitely generated'' and ``$\overline{\Bbbk}=\Bbbk$'',  then there are two kinds of pathological examples of operads
of GK-dimension 1:
\begin{enumerate}
\item[(a)](Example \ref{zzex6.1})
When $\dim \overline{\Bbbk}=\infty$, there is an example of
infinitely generated prime operad $\PP$ of GK-dimension 1 such that
every element in $\PP$ is central.
\item[(b)](Example \ref{zzex6.2})
Fix an arbitrary field $\Bbbk$, there is an example of
infinitely generated prime operad $\PP$ of GK-dimension 1 satisfying
\begin{enumerate}
\item[(bi)]
every nonzero element in $\PP(n)$ for $n\geq 2$ is not central,
\item[(bii)]
every finitely generated suboperad of $\PP$ is finite-dimensional and
$\PP$ is a union of finite-dimensional suboperads,
\item[(biii)]
there is an ascending chain of distinct ideals $\{I_i\}_{i\geq 0}$ of
$\PP$ such that every $\PP/I_i$ is an infinitely generated prime operad
of GK-dimension 1 satisfying (bi) and (bii).
\end{enumerate}
\end{enumerate}
Moreover, if we further remove the ``prime'' hypothesis, there is an operad
$\PP$ of GK-dimension 1 such that it is infinitely
generated and satisfies other properties (Example \ref{zzex6.3}).
In fact, very little is
known about general properties of infinitely generated prime
operads of GK-dimension 1 except for some easy examples.

This paper mainly concerns finitely generated operads. 
There are two families of
{\it saturated} (Definition \ref{zzdef-new-3.3}) prime suboperads of
$\Com$ and $\Ope$ respectively of GK-dimension 1: for $w\geq 1$,
\begin{equation}
\label{E0.3.1}\tag{E0.3.1}
\Com^{\{w\}}(n)=\begin{cases} 0 & n\neq w\lfloor\frac{n-1}{w}\rfloor+1,\\
\Bbbk \mu_n & n= w\lfloor\frac{n-1}{w}\rfloor+1,
\end{cases}
\quad {\text{and}} \quad \mu_n\ast \sigma =\mu_n, \; \forall \sigma\in \SG_n,
\end{equation}
and
\begin{equation}
\label{E0.3.2}\tag{E0.3.2}
\Ope^{\{2w\}}(n)=\begin{cases} 0 & n\neq 2w\lfloor\frac{n-1}{2w}\rfloor+1,\\
\Bbbk \mu_n &   n=2w\lfloor\frac{n-1}{2w}\rfloor+1,
\end{cases}
\quad {\text{and}} \quad \mu_n\ast \sigma =\sgn(\sigma)\mu_n, \; \forall
\sigma\in \SG_n.
\end{equation}
For $w>1$, $\Com^{\{w\}}$ (resp. $\Ope^{\{2w\}}$)
is a higher-arity version of $\Com$ (resp. $\Ope$), and clearly $\Com^{\{1\}}=\Com$ (resp.
$\Ope^{\{2\}}=\Ope$). Moreover,
$\Com^{\{w\}}$ and $\Ope^{\{2w\}}$ are Veronese powers of $\Com$
and $\Ope$ respectively in the sense of \cite{DMR20}.

 Now we can state our second result on classification of {\it semiprime} (Definition \ref{zzdef1.5}(2))
 saturated operads of GK-dimension 1.

\begin{theorem} [${\text{Classification of semiprime
 saturated operads of GK-dimension 1}}$]
\label{zzthm0.4}
Suppose $\Bbbk$ is algebraically closed. Every
finitely generated semiprime saturated operad of GK-dimension 1 is
isomorphic to a finite direct sum of operads of the form
\eqref{E0.3.1} or \eqref{E0.3.2}.
\end{theorem}

Recall that an operad $\PP$ is called {\it
uniformly bounded} if there is a finite number $N$ such that
$\dim \PP(n) \leq N$ for every~$n$. 
The next result concerns properties of finitely generated operads
of GK-dimension 1 over a field $\Bbbk$ with ${\text{char}}\; \Bbbk\neq 2$. 

\begin{theorem}
\label{zzthm0.5}
Suppose ${\text{char}}\; \Bbbk\neq 2$.
Let $\PP$ be a finitely generated operad of GK-dimension 1.
\begin{enumerate}
\item[(i)]
The Hilbert series of $\PP$ is rational, and thus $\PP$ is uniformly bounded.
\item[(ii)]
$\PP$ is right noetherian.
\end{enumerate}
\end{theorem}

In general, $\PP$ in Theorem \ref{zzthm0.5} is not left noetherian
\cite[Example 6.4(1)]{LQXZ25}.
It can be seen from \cite[Proof of Theorem 1.7(3)]{QXZZ20} that Theorem \ref{zzthm0.5}(i) fails if the GK-dimension is
$>1$. Surprisingly, Theorem \ref{zzthm0.5}
also fails for nonsymmetric operads, see
\cite[Construction 2.3 and Example 7.2]{QXZZ20}.   If
${\text{char}}\; \Bbbk= 2$, we suspect that Theorem
\ref{zzthm0.5} is still valid. However, our proof of Theorem
\ref{zzthm0.5} is based on \cite[Corollary 0.4]{LQXZ25}
where ${\text{char}}\; \Bbbk\neq 2$ is assumed.

We call~$\PP$ a {\it linear} operad if each $\PP(n)$ is 1-dimensional
over $\Bbbk$ for every~$n\geq 1$. We would like to mention an interesting
linear operad $\Mas$ which is a special case of Massey operads $\Mas^{a}_{b}$
\cite[Example 6.2]{LQXZ25} by taking $a=b=1$ and 
contains $\Ope$ as a suboperad.

\begin{example}
\label{zzex0.6}
Let $\Mas$ be the {\it Massey} operad defined by
\begin{enumerate}
\item[(i)]
$\Mas(n)=\Bbbk \mu_n$ for $n\geq 1$ and $\Mas(0)=0$,
\item[(ii)]
$\mu_n\ast \sigma=\sgn(\sigma) \mu_n$ for all $\sigma\in \SG_n$,
\item[(iii)]
for every $1\leq i\leq n$, we have
$$\mu_n\circ_i \mu_m=\begin{cases} 0, & {\text{if $n$ and $m$ are even,}}\\
(-1)^{i-1} \mu_{n+m-1}, &{\text{if $n$ is odd and $m$ is even,}}\\
\mu_{n+m-1}, & {\text{if $m$ is odd.}}
\end{cases}$$
\end{enumerate}
\end{example}

Operads of the form $\Lin_e(B)$ (in Corollary~\ref{zzcor-new-4.3}) or $\Lin_{o}(B)$ (in Remark~\ref{zzrema-new3-4.6}) are linear operads.  And we have the following

\begin{theorem} [${\text{Classification of saturated linear operads}}$]
\label{zzthm0.7}
Suppose ${\text{char}}\; \Bbbk\neq 2$.
Let $\PP$ be a finitely generated saturated linear operad.
Then $\PP$  is of the form 
$\Lin_e(B)$ or $\Lin_{o}(B)$, where~$B$ is nilpotent. 
 \end{theorem}

We believe that the above results about operads of
GK-dimension 1 are helpful for understanding some
aspects of operads of higher GK-dimensions. 
Here is a list of interesting questions.

\begin{question}
\label{zzque0.8}
\begin{enumerate}
\item[(1)]
How can we classify all prime finitely generated operads of
GK-dimension 2?
\item[(2)]
If $\PP$ is in part (1), is its Hilbert series rational?
\item[(3)]
What invariants can we define for operads of GK-dimension 2?
\end{enumerate}
\end{question}

The article is organized as follows: In Section 1 we recall the
partial definition of an operad and other basic material
for later sections.  Bergman's gap theorem for symmetric operads (Theorem \ref{zzthm0.1}) and Theorem \ref{zzthm0.5}
are proved in Section~\ref{zzsec-gap}.  Then in Section~\ref{zzsec-sat}, we introduce the notion of saturated operads and establish some properties on certain saturated PGC algebras (Definition~\ref{zzdef-new-3.5}) that will be useful in the sequel.
In Section~\ref{zzsec-cate-equi}, we develop a little bit of the category equivalence results  proved in \cite{LQXZ25}. Section~\ref{zzsec5} is devoted to the proof of classification
results, namely, Theorems \ref{zzthm0.3}, \ref{zzthm0.4} and
\ref{zzthm0.7}.  In Section~\ref{zzsec6} we give some comments,
remarks, examples, and questions. Finally, the appendix offers some results that are related to the main topic of the paper.

\section{Preliminaries}
\label{zzsec1}
This section contains some definitions and preliminary
materials that will be used in later sections. We first recall the partial definition \cite[Section 5.3.4]{LV12}
of a (symmetric) operad.

\begin{definition}
\label{zzdef1.1}
An \emph{operad} consists of the following data:
\begin{enumerate}
\item[(i)]
a sequence $\{\PP(n)\}_{n\geq 0}$ of right $\Bbbk\SG_n$-modules,
whose elements are called \emph{$n$-ary operations},
\item[(ii)]
the {\it arity} of an element $\nu\in \PP(n)$ is defined to be
$\Ar(\nu)=n$,
\item[(iii)]
an element $1_{\PP}\in \PP(1)$ called the \emph{identity},
\item[(iv)]
for all integers $m\ge 1$, $n \ge0$, and $1\le i\le m$, a
\emph{partial composition map}
\[-{\circ}_{i}-\colon \PP(m) \otimes \PP(n) \to
\PP(m+n-1), \]
\end{enumerate}
satisfying the following axioms:
\begin{enumerate}
\item[(a)]
for $\theta\in \PP(n)$ and $1\leq i\leq n$,
$\theta\circ_i  1_{\PP} = \theta =1_{\PP} \circ_1 \theta;
$
\item[(b)]
for $\lambda \in \PP(l)$, $\mu\in \PP(m)$ and $\nu\in \PP(n)$,
\begin{align}
\label{E1.1.1}\tag{E1.1.1}
(\lambda   \circ_i  \mu) \circ_{i-1+j}  \nu
&=\lambda  \circ_i  (\mu \circ_j  \nu),
\quad 1\le i\le l, 1\le j\le m,\\
\label{E1.1.2}\tag{E1.1.2}
(\lambda   \circ_i  \mu) \circ_{k-1+m} \nu
&=(\lambda  \circ_k \nu) \circ_i  \mu,
\quad 1\le i<k\le l;
\end{align}
\item[(c)]
for $\mu\in \PP(m)$, $\phi\in \SG_m$, $\nu\in \PP(n)$ and $\sigma\in \SG_{n}$,
\begin{align}
\label{E1.1.3}\tag{E1.1.3}
\mu  \circ_i  (\nu \ast \sigma)= &(\mu  \circ_i  \nu)\ast \sigma',\\
\label{E1.1.4}\tag{E1.1.4}
(\mu\ast \phi)  \circ_i  \nu=
&(\mu  \underset{\phi(i)}{\circ} \nu)\ast \phi'',
\end{align}
where
\begin{equation}\notag
\begin{split}
\sigma'= \vartheta_{m; 1, \cdots, 1, \underset{i}{n}, 1, \cdots, 1}
(1_m, 1_1, \cdots, 1_1, \underset{i}{\sigma}, 1_1, \cdots, 1_1),\\
\phi''= \vartheta_{m; 1, \cdots, 1, \underset{i}{n}, 1, \cdots, 1}
(\phi, 1_1, \cdots,1_1, \underset{i}{1_n}, 1_1 \cdots,  1_1).
\end{split}
\end{equation}
(see \cite[E8.0.1]{BYZ20} for the definition of
$\vartheta_{m; 1, \cdots, 1, \underset{i}{n}, 1, \cdots, 1}$, see also~\cite[Sec 5.3.4]{LV12} ).
\end{enumerate}
\end{definition}
 
\begin{definition}
\label{zzdef1.2}
The \emph{Gelfand-Kirillov dimension} (or \emph{GK-dimension}
for short) of a locally finite operad $\PP$ is defined to be
\begin{equation}
\label{E1.2.1}\tag{E1.2.1}
\GK(\PP):=\limsup_{n\to\infty} \;
\log_n\left(\sum_{i=0}^n\dim  \PP(i)\right).
\end{equation}
\end{definition}

The GK-dimension of a locally finite graded algebra is defined analogously. In the case the graded algebra is finitely generated, it is straightforward to show that this definition coincides with the classical definition of the GK-dimension of associative algebras.

The {\it Hilbert series} (also called {\it generating series}) of
$\PP$ is defined to be the formal power series \cite[(0.1.2)]{KP15}
\begin{equation}
\label{E1.2.2}\tag{E1.2.2}
H_{\PP}(t):=\sum_{n=0}^{\infty} \dim \PP(n) \; t^n.
\end{equation}
In \cite{KP15} and \cite{QXZZ20} it is denoted by $G_{\PP}(t)$.
The Hilbert series of a graded vector space or a graded algebra
is defined in the same way.

Let $\PP$ be an operad and $w$ be an  integer $\geq 2$.
Recall from \cite{LQXZ25} that the suboperad $\PP_{\{w\}}$
of $\PP$ is defined as
\begin{equation}
\label{E1.2.3}\tag{E1.2.3}
\PP_{\{w\}}(n)=\begin{cases}
\Bbbk 1_{\PP} & n=1,\\
0& 2\leq n\leq w-1,\\
\PP(n) & n\geq w.\end{cases}
\end{equation}
For an $\NN$-graded algebra $A$, $A_{\{w\}}$ can be defined similarly.

Recall that two operads $\PP$ and $\QQ$ are {\it almost isomorphic}
if there is an integer $w$ such that $\PP_{\{w\}}$ and $\QQ_{\{w\}}$
are isomorphic.

We say $\PP$ is {\it left noetherian} (resp. {\it right
noetherian}) if the left {\rm{(}}resp. right{\rm{)}}
ideals of $\PP$ satisfy the ascending chain condition.

Let $I$ and $J$ be two subcollections of an operad $\PP$. Then
$I\circ J$ is the $\SG$-submodule of~$\PP$ generated by
$x\circ_i y$ for all $x\in I$, $y\in J$ and $1\leq i\leq \Ar(x)$.
We define a prime (or semiprime) operad as follows.

\begin{definition}
\label{zzdef1.5}
Let $\PP$ be an operad.
\begin{enumerate}
\item[(1)]
We say $\PP$ is {\it prime} if for all nonzero ideals $I$
and $J$ of $\PP$, $I\circ J\neq 0$.
An ideal $I$ is called {\it prime} if $\PP/I$ is a
prime operad.
\item[(2)]
We say $\PP$ is {\it semiprime} if the intersection of all
prime ideals of $\PP$ is 0.
\end{enumerate}
\end{definition}

Let $\PP$ be an operad (which is reduced in this article). We
define a unital ${\mathbb N}$-graded associative algebra
$A:=(\oplus_{i=0}^{\infty} A_i, \cdot)$ by
\begin{enumerate}
\item[(i)]
$A_i=\PP(i+1)$, for all $i\geq 0$, and
\item[(ii)]
$x\cdot y= x\circ_1 y$ for all homogeneous elements in $A$.
\end{enumerate}
By Definition \ref{zzdef1.1}(a) and by \eqref{E1.1.1} with $i=1=j$, $A$ is a unital graded
associative algebra. We denote it by $A_{\PP}$. We now define
a functor ${\mathcal F}$ from the category of operads,
denoted by $\Op$, to the category of unital ${\mathbb N}$-graded
associative algebras, denoted by $\NGAss$ as follows: for every
$\PP\in \Op$, define~${\mathcal F}(\PP)=A_{\PP}$, and ${\mathcal F}(\varphi)=\varphi$  for every morphism $\varphi$ from $\PP$ to $\QQ$ in $\Op$. It is clear that~${\mathcal F}$ is a functor from $\Op$ to $\NGAss$
\cite[Definition 3.1]{LQXZ25}. By definition, we have
$H_{\PP}(t)=t H_{A_{\PP}}(t)$. Consequently, ${\mathcal F}$
preserves the GK-dimension and rationality of the Hilbert series.

We shall frequently apply the equivalence of categories of certain operads and certain graded associative algebras in the sequel. And all algebras involved in this article are associative
(with possibly additional structures).

We conclude this section with a classical result of Dickson on $\SG_n$-modules.

\begin{lemma} [\cite{Di08}]
\label{zzlem1.3}
Let $n\geq 9$. Suppose $M$ is a simple right $\SG_n$-module
of dimension $<n-2$. Then $M$ is one dimensional.
\end{lemma}

Also see \cite[Theorem (1.1)]{Wag76} and \cite[Theorem 1.1]{Wag77}
for related results. It is well-known that there exist only two
1-dimensional $\SG_n$-modules, namely, the trivial representation,
denoted by $\triv$, and the sign representation, denoted by $\sign$.

Let~char $\Bbbk=p $.
If $n\geq 5$ and $p=2$,
by \cite[Lemma 1.3]{LQXZ25},
$\Ext^1_{\Bbbk \SG_n}(\triv, \triv)=\Bbbk$,
and denote by $\extn$  the 2-dimensional indecomposable right $\SG_n$-module
corresponding to any nonzero element in
$\Ext^1_{\Bbbk \SG_n}(\triv, \triv)$.

\begin{proposition}
\label{zzpro1.4}
Let $M$ be a right $\SG_n$-module of dimension
$<n-2$ with~$n\geq 9$.
\begin{enumerate}
\item[(1)]
Suppose $p\neq 2$. Then $M$ is a direct sum of copies of $\triv$
and $\sign$. As a consequence, $M(1-g)=0$ for all $g\in \AG_n$.
\item[(2)]
Suppose $p=2$. Then $M(1-g)=0$ for all $g\in \AG_n$. As a
consequence, $M$ is a direct sum of copies of $\triv$ and
$\extn$.
\end{enumerate}
\end{proposition}

\begin{proof} (1) By Lemma \ref{zzlem1.3}, $M$ has a composition
series such that the composition factors consist of $\triv$ and $\sign$ only. By
\cite[Lemma 1.3(1)]{LQXZ25}, $M$ is a direct sum of copies of
$\triv$ and $\sign$. Since $\triv(1-\sigma)=\sign(1-\sigma)=0$ for all
$\sigma\in \AG_n$, the consequence follows.

(2) By Lemma \ref{zzlem1.3}, $M$ has a composition series
such that the composition factors consist of $\triv$ only. We prove the assertion by induction
on the dimension of $M$. If~$\dim M=1$, it is trivial. Now assume
that the assertion holds for all $M'$ with~$\dim M'\leq m-1$ and
let $M$ have dimension $m$. Let $M_0$ be a 1-dimensional submodule
with basis element $x_1$ and $M/M_0$ has a basis
$\{\bar y_1,\cdots,\bar y_{m-1}\}$. By induction hypothesis,
$\AG_n$-actions on $M_0$ and $M/M_0$ are trivial. Therefore,
for every $\sigma\in \AG_n$, we have
$$\begin{aligned}
x_1\ast \sigma &=x_1,\\
y_i\ast \sigma &=y_i+f_i(\sigma) x_1.
\end{aligned}
$$
It is easy to see that each $f_i$ is a group homomorphism from
$\AG_n$ to $(\Bbbk,+)$. Since $\AG_n$ is simple, $f_i=0$ for all
$i$. This means that the $\AG_n$-action on $M$ is trivial, as required.

Now $M$ is a right module over $\Bbbk (\SG_n/\AG_n)$. As a module
over $\Bbbk (\SG_n/\AG_n)$, $\extn$ is injective. This
implies that $M$ is a direct sum of copies of $\triv$ and $\extn$.
\end{proof}

The above proposition will be used in Section~\ref{zzsec-gap}
as one of the key steps.

\section{Bergman's gap theorem for symmetric operads}
\label{zzsec-gap}

Recall that operads are always locally finite in this paper and an operad is called \emph{connected} if~$\dim \mathcal{P}(1)=1$.  We
say that an operad $\PP$ has {\it sub-quadratic growth} if
$$\limsup_{n\to\infty}
\frac{\sum_{i=0}^n \dim \PP(i)}{n^2}
=0.$$
It is easy to see that if $\GK(\PP)<2$, then $\PP$ has
sub-quadratic growth. We shall see that the converse is also true for finitely generated operads (Theorem \ref{zzthm-new-2.1}).
The main result of this section is the following Bergman's gap
theorem for finitely generated operads which is slightly
stronger
than Theorem \ref{zzthm0.1}. Both symmetric and
nonsymmetric operads will be involved in the proof of Theorem~\ref{zzthm-new-2.1}.

\begin{theorem}
\label{zzthm-new-2.1}
Let $\PP$  be   a finitely generated symmetric operad that
has sub-quadratic growth. Then $\GK(\PP)\leq 1$.
\end{theorem}

We need a few steps of reductions and then use the gap theorem for
nonsymmetric operads \cite[Theorem 1.1]{QXZZ20}.

By \cite[Lemma 1.4]{LQXZ25}, we can
assume $\PP=\PP_{\{w\}}$ for some $w\geq 2$. In particular,  $\PP(0)=0$ and $\PP(1)=\Bbbk$. Let $\VV$ be a finite
dimensional $\SG$-module that generates $\PP$ as a symmetric
operad. Then $\VV\subseteq \sum_{i=2}^t \PP(i)$ where $t$ is the
largest arity of nonzero elements in $\VV$. Without loss of
generality, we can assume that $\VV=\sum_{i=2}^t \PP(i)$. When the $\SG$-action is forgotten, $\PP$ becomes a connected nonsymmetric operad (which is not necessarily finitely generated by  $\VV$ as a nonsymmetric operad).
Let $\QQ$ be the nonsymmetric suboperad
of $\PP$ generated by $\VV$. Note that $\QQ\subseteq \PP$ and that
$f\ast \sigma$ could be outside of $\QQ$ for $f\in \QQ(n)$ and
$\sigma\in \SG_n$. By equivariance axiom (Definition \ref{zzdef1.1} (c))
of a symmetric operad, we have
$$
\QQ(n)\subseteq \QQ(n)\ast (\Bbbk   \SG_n)=\PP(n)
$$
for all $n$.

Given two subcollections $\VV_1$ and $\VV_2$ of $\QQ$, let $\VV_1
\circ_{ns} \VV_2$ be the subcollection of $\QQ$
spanned by all elements of the form $v \circ_i w$ for $v\in \VV_1$,
$w\in \VV_2$ and $1\leq i\leq \Ar(v)$. (This is slightly different
from $\VV_1\circ \VV_2$ for subspaces $\VV_1$ and $\VV_2$ in a
symmetric operad). Given
a subcollection $\WW$ of $\QQ$, let $\WW^0=(0,\Bbbk  ,0,0,\dots)$,
and inductively, let $\WW^m=\WW^{m-1}\circ_{ns} \WW$ for $m\geq1$.

 If $\PP$ has sub-quadratic growth, then so does $\QQ$. The
following lemma says that we can apply \cite[Theorem 5.12]{QXZZ20} to $\QQ$. 

\begin{lemma}
\label{zzlem-new-2.2}
Retain the notation as above. 
\begin{enumerate}
\item[(1)]
For all fixed integers $a_1,a_2, a_3>1 $, there exists
$b_1\geq a_1$ such that
$$\dim \PP(n)< \frac{1}{a_3} n$$
for all $b_1< n\leq b_1+a_2$.
As a consequence, the assertion 
holds true for $\QQ$ as well.
\item[(2)]
There is an integer $d\geq 3$ such that
$$\dim \left( \left(\sum_{i=0}^d \VV^i\right)/
\left(\sum_{i=0}^{d-1} \VV^i\right)\right) \leq d-3.$$
As a consequence $\GK(\QQ)\leq 1$.
\item[(3)]
If the right $\AG_n$-action on $\QQ(n)$ is trivial for
$n\gg 0$, then $\GK(\PP)\leq 1$.
\item[(4)]
If the right $\AG_n$-action on $\QQ(n)$ is trivial for
$n\gg 0$, then $\PP(n)=\QQ(n)$ for~$n\gg 0$.
Consequently, $\PP$ is finitely generated as a nonsymmetric
operad.
\end{enumerate}
\end{lemma}

\begin{proof} (1) Suppose to the contrary that the assertion fails. 
Then for each $m\geq 1$, there is an integer~$d$ between~$a_1+ma_2 +1
(=:b_1+1)$ and $a_1+(m+1)a_2(=:b_1+a_2)$ such that
$\dim \PP(d)\geq \frac{1}{a_3} d> \frac{1}{a_3} b_1
\geq \frac{a_2}{a_3} m$.  For every $n\gg 0$, let
$M=\lfloor \frac{n-a_1}{a_2} \rfloor-1$. Then $n\geq (M+1)a_2+a_1$.
Thus, for every $n\gg 0$, we have
$$\begin{aligned}
\sum_{i=0}^{n} \dim \PP(i)
&\geq \sum_{m=1}^{M} \left(\sum_{d=a_1+ma_2+1}^{a_1+(m+1)a_2}
\dim \PP(d)\right)
 \geq \sum_{m=1}^{M} \frac{a_2}{a_3} m&\\
&\geq \frac{a_2}{2a_3}M(M+1)\sim \frac{1}{2a_2a_3} n^2.&
\end{aligned}
$$
This contradicts the hypothesis that $\PP$ has sub-quadratic
growth. Therefore the assertion follows. The consequence is clear since $\QQ\subseteq \PP$.

(2)
Note that $\QQ$ (as a suboperad of $\PP$) has sub-quadratic growth.
By replacing $\PP(n)$ by 
$(\sum_{i=0}^d \VV^i)/(\sum_{i=0}^{d-1} \VV^i)$
in the proof of part (1), 
 one obtains that
$\dim \left( (\sum_{i=0}^d \VV^i)/(\sum_{i=0}^{d-1} \VV^i)\right) \leq
\frac{1}{2} d$ for some $d>6$ by taking $a_1=6$ and~$a_3=2$.  The
assertion follows from the fact that $\frac{1}{2} d< d-3$.
The consequence follows from \cite[Theorem 5.12]{QXZZ20}.

(3) If the right $\AG_n$-action on $\QQ(n)$ is trivial, then,
for every~$\sigma\in \SG_n\setminus \AG_n$, we have
\begin{equation}
\label{E2.2.1}\tag{E2.2.1}
\PP(n)=\QQ(n)\ast (\Bbbk   \SG_n)
=\QQ(n)\ast (\Bbbk  \AG_n+\Bbbk   \AG_n \sigma)
=\QQ(n)+\QQ(n)\ast \sigma
\end{equation}
which implies that
$$\dim \QQ(n)\leq \dim \PP(n)\leq 2 \dim \QQ(n).$$
Since the above inequalities hold for $n\gg 0$,
$\GK(\PP)=\GK(\QQ)\leq 1$.

(4) 
Without loss of generality, we can assume that $\PP$ is
generated by an $\SG$-module $\VV$ with arities between 2
and $t$. For every tree monomial $T\in \QQ(n)$ with $n \geq 2$, by \eqref{E1.1.3} we can always find $\sigma_T=(k,k+1)$ such
that $T\ast \sigma_T \in \QQ(n)$. 
Since the right $\AG_n$-action on $\QQ(n)$ is trivial for $n\gg 0$, then for every tree monomial $T\in \QQ(n)$ and every~$\sigma\in \SG_n\setminus \AG_n$, we have
$T\ast \sigma = T\ast \sigma_T \in \QQ(n)$ when $n\gg 0$.
Noting that every element in $\QQ(n)$ is a linear combination
of tree monomials, we know that $\QQ(n)\ast \sigma \subseteq \QQ(n)$ for every~$\sigma\in \SG_n\setminus \AG_n$ and $n\gg 0$.
By \eqref{E2.2.1}, $\PP(n)=\QQ(n)$ when $n\gg 0$.
The consequence is clear.
\end{proof}

Now we recall some conventions from \cite[Section 2]{QXZZ20}:  Let
$\XX$ be a finite operation alphabet and let $\QQ=\TX/\II$,
where~$\TX$ is the free nonsymmetric operad generated by~$\XX$  and $\II$ is an ideal of~$\TX$. 
Denote by~$\Irr(\II)$ the set of tree monomials that are not divisible by the leading monomial of any element in~$\II$. In particular,  ~$\Irr(\II)$ is a $\Bbbk$-linear basis of~$\QQ$ and 
$\QQ$ is generated by $\VV=\Bbbk  \XX$. 
We denote by~$\wt(T)$ the weight of a tree monomial $T$.  
We recall that a tree monomial~$v$ \emph{single-branched} if $v$ is of the form
$$v=x_{i_1}\circ_{j_1}(x_{i_2}\circ_{j_2} (\cdots (x_{i_m}\circ_{j_m}x_{j_{m+1}})\cdots ))$$
with~$x_{i_1},\dots,x_{i_{m+1}}$ lies in~$\XX$;  moreover, we call $x_{j_{m+1}}$ the \emph{top} of~$v$ and $x_{i_1}$ the \emph{bottom} of~$v$; finally, $v$ is called \emph{toward to the rightmost} if~$j_p=\Ar(x_{i_p})$ for every~$1\leq p\leq m$. 

By Lemma \ref{zzlem-new-2.2},  in order to prove Theorem~\ref{zzthm-new-2.1},  it suffices to show that the right~$\AG_n$-action on $\QQ(n)$ is trivial for $n\gg 0$.
Part (1) of the following
proposition is an intermediate step.

\begin{proposition}
\label{zzpro-new-2.3}
Suppose $\PP$ is a finitely generated
symmetric operad that has sub-quadratic growth.
Let $\VV=\Bbbk\XX$ be a generating subspace of $\PP$ and $t$ be the largest arity of nonzero elements in $\VV$. Let $\QQ$ be the nonsymmetric operad generated by~$\VV$ and assume that $\QQ=\TX/\II$. 
\begin{enumerate}
\item[(1)]
For all fixed integers $a_1,a_2>1 $, there is an integer~$b_1\geq a_1$
such that the right $\AG_n$-action on $\QQ(n)$ is trivial
for all $b_1< n\leq b_1+a_2$.
\item[(2)]
There are integers $N,M,w$, only dependent on $\PP$, such that
if $T\in\Irr(\II)$ with $\Ar(T)\geq N$ then  
$T$ has the form
\begin{equation}\label{E2.3.1}\tag{E2.3.1}
T= v_1 \circ_{i_1} ( v \circ_{i_2} (x\circ_j  v_3)), 
\end{equation}
where $v_1, v_3\in \sum_{d=M}^{M+t} \QQ(d)$, $x\in \XX$, $1\leq j\leq \Ar(x)$, $v_2:=v \circ_{i_2} x$  is
single-branched and periodic of period $w$; 
\item[(3)]
In statement (2), we can always replace 
$(M,N)$ by $(M', N+2(M'-M))$ for any $M'>M$.
\end{enumerate}
\end{proposition}

\begin{proof} (1) We can always assume that $b_1\geq 9$.
By Lemma \ref{zzlem-new-2.2}(1), taking $a_3=2$, we have
$\dim \PP(n)<\frac{1}{2} n< n-2$ for all $b_1< n \leq b_1+a_2$.
By Proposition \ref{zzpro1.4}, the right $\AG_n$-action on
$\PP(n)$ is trivial. Therefore, the right $\AG_n$-action on
$\QQ(n)$ is trivial for all $b_1< n\leq b_1+a_2$.

(2) Denote   $n:=\wt(T)$.   By Lemma
\ref{zzlem-new-2.2}(2) and  by the proof of \cite[Lemma 5.11]{QXZZ20},
there are positive integers~$N_1,M_1\geq 2$, such that for every~$n\geq N_1$,
$T$ has a decomposition   $T= v_1^\prime \circ_{j_1}
(v^\prime \circ_{j_2'} (x\circ_{j_2''} v_3^\prime))$,   such that
$\wt(v_1^\prime),\wt(v_3^\prime)\leq M_1$,   $x\in \XX$ and  $v_2^\prime:=v^\prime \circ_{j_2'} x$  
is single-branched and periodic of period $w$. 
Note that, by \cite[Claim 5.6]{QXZZ20}, the minimal period of a single-branched monomial in $\Irr(\II)$ is not bigger than a fixed integer (say $d$) only depending on $\PP$,
then  $w:=(d-1)!$ is a common period of all single-branched monomials in $\Irr(\II)$,  
in particular, $w$ is not dependent on~$T$.  And  we have~$\Ar(v_1^\prime),\Ar(v_3^\prime)\leq M_1t$ and
$\wt(v_2^\prime)\geq n-2M_1$.

Take $N=N_1t+2M_1t^2$ and $M=M_1t-t$. If $\Ar(T)\geq N$, then~$n\geq \frac{1}{t}\Ar(T)\geq N_1+2M_1t$ and thus there
is a decomposition $T= v_1^\prime \circ_{j_1}
(v_2^\prime \circ_{j_2} v_3^\prime)$. If $\Ar(v_3^\prime)<M$,
define a new decomposition of $T$ by moving 
the element on the
top of $v_2^\prime$ to the bottom of $v_3^\prime$, i.e., write
$v_2^\prime =v_2^{\prime\prime}\circ_{j_2^\prime}x$ with
$\wt(x)=1$ and define $v_3^{\prime\prime}
=x\circ_{j_2^{\prime\prime}} v_3^{\prime}$  for some suitable
$j_2^{\prime\prime}$, then $T=v_1^\prime \circ_{j_1}
(v_2^{\prime\prime} \circ_{j_2^\prime} v_3^{\prime\prime})$ and
$\Ar(v_3^\prime)<\Ar(v_3^{\prime\prime})\leq\Ar(v_3^\prime)+t$.
We can continue this procedure until, in the new decomposition,
$M\leq\Ar(v_3^{\prime\prime})\leq M+t$. In a similar way, we
can define $v_1^{\prime\prime}$ such that
$M\leq\Ar(v_1^{\prime\prime})\leq M+t$. So we obtain a
decomposition as required.

(3) It follows from the reasoning of part (2) that we can increase the arities of~$v_1$ and~$v_3$ in the expression~$T= v_1 \circ_{i_1} ( v \circ_{i_2} (x\circ_j  v_3)) $ by moving the bottom of $v_2:=v\circ_{i_2}x $ to $v_1$ and moving the top of $v_2 $ to  $v_3$. 
So the result follows.
\end{proof}

We are now ready to prove the main theorem of this section.

\begin{proof}[Proof of Theorem \ref{zzthm-new-2.1}] 
By Lemma \ref{zzlem-new-2.2}(3), it remains to show that the right
$\AG_n$-action on $\QQ(n)$ is trivial for $n\gg 0$. We write
$\QQ$ as a quotient operad ${\mathcal T}(\XX)/\II$, where $\VV=\Bbbk\XX$
is a subcollection of $\PP$. So there is an
$\SG_n$-action on $\VV(n)$.
Each element in $\QQ(n)$ is a linear combination of tree monomials $T\in \Irr(\II)\subseteq {\mathcal T}(\XX)$. We use the monomial order on ${\mathcal T}(\XX)$ defined by path
sequences, see \cite[Section 3]{QXZZ20}. 

We claim that $T\ast \sigma=T$ for each tree monomial $T$
considered as an element in~$\QQ(n)$ and~$\sigma\in \AG_n$ with $n\gg 0$. 

Let $T=v_1 \circ_{i_1} ( v \circ_{i_2} (x\circ_j  v_3))$ be as in~\eqref{E2.3.1} and let~$a_1=M$, $a_2=(w+2)t$. By Proposition \ref{zzpro-new-2.3}(1),
there is an integer~$b_1\geq a_1=M$ such that the right $\AG_n$-action on
$\QQ(n)$ is trivial for all~$b_1<n\leq b_1+a_2$. By Proposition
\ref{zzpro-new-2.3}(3), if~$M$ is smaller than $b_1$, we can replace~$(M,N)$ by $(b_1, N+2(b_1-M))$. Therefore, we can further assume that~$M=b_1$ and thus the right $\AG_n$-action on $\QQ(n)$ is trivial for all~$M< n \leq M+(w+2)t$.

Recall that
$\AG_n$ is generated by 3-cycles
$\{\sigma_k:=(k,k+1,k+2)\}_{k=1}^{n-2}$. So it suffices to prove the claim for~$\sigma=\sigma_k$, $1\leq k\leq n-2$.

\textbf{Case 1}:
Suppose~$T=v_1\circ_{\Ar(v_1)}(v\circ_{\Ar(v)} (x\circ_{j}v_3))$,  where $x\in \XX$, $1\leq j\leq \Ar(x)$, and~$v_2=v\circ_{\Ar(v)} x$ is
toward to the rightmost.
Then the action of $\sigma_k$ will not involve~$v_3$ and~$v_1$  simultaneously.   
Thus every $\sigma_k$ acts on a sub-tree monomial of~$T$ either involving~$v_1$ (or $v_3$) and a sub-tree monomial of $v_2$ with weight~$\leq 2$, or involving only a sub-tree monomial of $v_2$ with weight~$\leq 3$. It follows that $\sigma_k$ acts on a sub-tree monomial of $T$ which
has arity between~$M$ and~$M+3t\ (\leq M+(w+2)t)$. So we have $T\ast \sigma_k=T$. 

\textbf{Case 2}:
Suppose $T=v_1 \circ_{i_1} ( v \circ_{i_2} (x\circ_j  v_3))$ is not of the form described in Case~1. 
Since~$v\circ_{i_2}x$ is single-branched, we may assume that $v_1\circ_{i_1}(v\circ_{i_2}x)= (v_1\circ_{i_1} u)\circ_{i_1'} v_2'$  for some sub-tree monomial $u$ of $v\circ_{i_2}x$
such that $v_1'':=v_1\circ_{i_1} u$ is of the smallest possible weight satisfying~$i_1'\neq \Ar(v_1'')$. Since~$v\circ_{i_2}x$ is single-branched with large enough weight  and periodic with period~$w$,  we deduce that~$\wt(v_1)\leq \wt(v_1'')\leq \wt(v_1)+(w+1)$ by the choice of~$v_1''$.   It follows that
$$T= v_1 \circ_{i_1} ( v \circ_{i_2} (x\circ_j  v_3)) 
=v_1''\circ_{i_1'}(v_2'\circ_{j_1} v_3).$$ 
In this case, we have~$\Ar(v_1)\leq \Ar(v_1'')\leq \Ar(v_1)+(w+1)t\leq M+(w+2)t$. Let~$m=\Ar(v_1'')\neq i_1'$. Then by construction, the right $\AG_m$-action on
$v_1''$ is trivial. Let~$\phi\in \AG_m$ such
that $\phi(i_1')=m$. Since $v_1''=v_1''\ast \phi$, we have
$$T=(v_1''\ast \phi)\circ_{i_1'}(v_2'\circ_{j_1} v_3)
=[v_1''\circ_{\phi(i_1')}(v_2'\circ_{j_1} v_3)]\ast \phi''
=[v_1''\circ_{m}(v_2'\circ_{j_1} v_3)]\ast \phi''$$
for some $\phi''\in \SG_n$, where the second equation
follows from \eqref{E1.1.4}.
As an element in~$\QQ$, $T':=v_1''\circ_{m}(v_2'\circ_{j_1} v_3)$ is
smaller than $T$ since $i_1'< m$.  Rewrite~$T'$ as a linear combination of smaller elements $T_i''$ of the form as in Proposition~\ref{zzpro-new-2.3}(2).
If some $T_i''$ is still not of the form in Case 1, we continue the above two processes, which must finally terminate because the monomial order is a well-order. If~$T''$ is of the form in Case~1, then we have~$T''\ast \sigma=T''$ for every~$\sigma\in \AG_n$.  

By the above reasoning, we may assume that~$T=T''\ast \tau$ for some element~$\tau\in \SG_n$ such that~$T''\ast \sigma=T''$ for every~$\sigma\in \AG_n$.  
Then
$$T\ast \sigma =(T''\ast \tau)\ast \sigma
=[T''\ast(\tau\circ \sigma\circ \tau^{-1})]\ast
\tau=T''\ast \tau=T,$$
which proves the claim.  The proof of Theorem \ref{zzthm-new-2.1} is completed.
\end{proof}

In fact we have proved the following.

\begin{theorem}
\label{zzthm-new-2.4}
Let $\PP$  be a finitely generated (symmetric) operad.
\begin{enumerate}
\item[(1)]
If~$\PP$ has sub-quadratic growth, then
$\GK(\PP)\leq 1$ and $\PP$ is almost $\AG\triv$.
\item[(2)]
If $\GK(\PP)=1$, then $\PP$ is almost $\AG\triv$.
\end{enumerate}
\end{theorem}

\begin{proof} (1) By Theorem~\ref{zzthm-new-2.1}, $\GK(\PP)\leq 1$.  Moreover, for all~$f\in \QQ(n)$, $\sigma\in \SG_n\setminus \AG_n$ and~$\phi\in \AG_n$ with $n\gg 0$,  by {\bf Claim 1} in the proof of
Theorem \ref{zzthm-new-2.1}, we have
 $$
(f\ast \sigma)\ast \phi= f\ast (\sigma\circ \phi\circ \sigma^{-1}) \ast \sigma=f\ast \sigma.
$$
By~\eqref{E2.2.1},  $\PP$ is almost $\AG\triv$.

(2) This is an immediate consequence of part (1).
\end{proof}

\begin{proof}[Proof of Theorem \ref{zzthm0.5}]
By Theorem~\ref{zzthm-new-2.4}(2) and~\cite[Corollary 0.4]{LQXZ25},  the result follows.
\end{proof}

\section{Saturation}
\label{zzsec-sat}
The aim of this section is to introduce the notion of saturated operads and establish some properties on certain saturated PGC algebras (Definition~\ref{zzdef-new-3.5}).
We first recall the definition of torsion elements in a graded
algebra and recall some results from~\cite{LQXZ25}.

Let~$A:=\oplus_{i=0}^{\infty} A_i$ be a locally finite
${\mathbb N}$-graded algebra. An element $x\in A$ is called
{\it right torsion} if $x A_{\geq n}=0$ for some $n\gg 0$
\cite[p. 233]{AZ94}. Similarly we can define a {\it left torsion}
element. The set of all right (resp. left) torsion elements of
$A$ is denoted by~$\tau^r(A)$ (resp. $\tau^l(A)$).
The left and right torsions of operads can be defined similarly,
see below.

\begin{definition}[{\cite[Definition 2.1]{LQXZ25}}]
Let $\PP$ be an operad and denote
$\oplus_{i\geq n}\PP(i)$ by~$\PP_{\geq n}$.
\begin{enumerate}
\item[(1)]
The {\it left torsion ideal} of $\PP$ is defined to be
$$\tau^l(\PP)=\{ x\in \PP\mid \PP_{\geq n} \circ (\Bbbk x)=0,
{\text{ for $n\gg 0$}}\}.$$
If $\tau^l(\PP)=0$, then $\PP$ is called {\it left
torsionfree}.
\item[(2)]
 The {\it right torsion ideal} of $\PP$ is defined to be
$$\tau^r(\PP)=\{ x\in \PP\mid (\Bbbk x)\circ \PP_{\geq n}=0,
{\text{ for $n\gg 0$}}\}.$$
If $\tau^r(\PP)=0$, then $\PP$ is called {\it
right torsionfree}.
\end{enumerate}
\end{definition}

\begin{lemma}
\label{zzlem-new-3.2}
Let $\PP$ be a locally finite operad.
\begin{enumerate}
\item[(1)]\cite[Lemma 1.4(4)]{LQXZ25}
$\PP_{\{w\}}$ is left noetherian if and only if so is $\PP$.
\item[(2)]\cite[Lemma 2.2(1)]{LQXZ25}
$\tau^l(\PP)$ is a 2-sided ideal of $\PP$ and $\tau^r(\PP)$ is
a right ideal of $\PP$.
\item[(3)]\cite[Lemma 2.2(4)]{LQXZ25}
Every finite dimensional left ideal of~$\PP$ is a subspace of
$\tau^l(\PP)$.
\item[(4)]\cite[Lemma 2.2(6)]{LQXZ25}
If $\PP$ is left noetherian, then $\tau^l(\PP)$ is finite
dimensional.
\item[(5)]\cite[Lemma 2.5(1)]{LQXZ25}
Suppose $\PP$ is semiprime. Then $\tau^l(\PP)\subseteq \PP(1)$.
As a consequence, if $\PP$ is prime and infinite dimensional,
then $\tau^l(\PP)=0$.
\end{enumerate}
\end{lemma}

 Now we are ready to introduce the concept of a saturated operad. Recall that operads are always assumed to be locally finite in this paper.

\begin{definition}
\label{zzdef-new-3.3}
Let $\PP$ and $\PP'$ be left torsionfree operads.
\begin{enumerate}
\item[(1)]
An operad $\PP$ is called {\it saturated} if
for every left torsionfree operad $\QQ$ with isomorphism
$\phi:\QQ_{\{w\}}\to \PP_{\{w\}}$ for some $w> 0$, the map
$\phi$ extends to an injective (operadic) homomorphism $\phi': \QQ \to \PP$.
\item[(2)]
A  homomorphism $f: \PP\to \PP'$ is called a {\it saturation} of $\PP$ if
\begin{enumerate}
\item[(2a)]
$f|_{\PP_{\{w\}}}:
\PP_{\{w\}}\to \PP'_{\{w\}}$ is an isomorphism for some $w\geq 2$, and
\item[(2b)]
$\PP'$ is saturated.
\end{enumerate}
Sometimes we simply call $\PP'$ a saturation of $\PP$
and write $\PP'=\overline{\PP}$.
\end{enumerate}
\end{definition}

The concepts of ``saturated'' and ``saturation'' for left torsionfree
$\mathbb{N}$-graded algebras are defined analogously.
It is easy to see that saturation of $\PP$ is unique
up to  isomorphism. We will show that $\Com^{\{w\}}$
\eqref{E0.3.1} and $\Ope^{\{2w\}}$ \eqref{E0.3.2} are
saturated in Corollary \ref{zzcor-new3-5.3}.

\begin{proposition}
\label{zzpro-new-3.4}
Let $\PP_1$ and $\PP_2$ be left
noetherian operads.
\begin{enumerate}
\item[(1)]
If $\PP_1$ and $\PP_2$ are saturated operads, then so is
$\PP_1\oplus \PP_2$.
\item[(2)]
If $\QQ_1\to \PP_1$ and $\QQ_2\to \PP_2$ are saturations, then
so is $\QQ_1\oplus \QQ_2\to \PP_1\oplus \PP_2$.
\end{enumerate}
\end{proposition}

\begin{proof} (1) First of all it is easy to see that $\PP'':=
\PP_1\oplus \PP_2$ is left torsionfree. Secondly it is
easy to check that $\PP''$ is left noetherian. It
remains to show the condition in Definition \ref{zzdef-new-3.3}(1).
Let $\QQ$ be a left torsionfree operad such that $\QQ_{\{w\}}
=\PP''_{\{w\}}$ for some $w$.  By Lemma \ref{zzlem-new-3.2}(1), $\QQ_{\{w\}}$ is left noetherian and so is $\QQ$.
Let $i$ be 1 or 2 and let 
$\Phi_i$ be the set of two-sided ideals $L$ of $\QQ$ such that $L_{\geq w'}=(\PP_i)_{\geq w'}$ for some $w'\geq w$.
Since $\QQ$ is noetherian, $\Phi_i$ contains a maximal element.
Let $I$ be a maximal ideal in $\Phi_1$ and $J$ be a maximal element in $\Phi_2$. Then we have~$(I\cap J)_{\geq m}=0$ for all integers~$m \gg 0$.
Consequently, $I\cap J$
is finite dimensional. Since $\QQ$ is left torsionfree, by Lemma \ref{zzlem-new-3.2}(3), we obtain
$I\cap J=0$.  Consider the
exact sequence
$$0\to I \to \QQ \to \QQ/I\to 0.$$
Since $I$ is maximal with respect to the described property,
$\QQ/I$ has no finite dimensional nontrivial ideals. Since~$\QQ$ is  left noetherian, so is $\QQ/I$. By Lemma
\ref{zzlem-new-3.2}(4), $\QQ/I$ is left torsionfree. By the
definition of $I$, $(\QQ/I)_{\geq w'}=(\PP_2)_{\geq w'}$.
Since~$\PP_2$ is saturated, there is an injective operadic
homomorphism $f: \QQ/I\to \PP_2$. By symmetry, there is
an injective homomorphism $g: \QQ/J\to \PP_1$. Since
$I\cap J=0$, we obtain an injective homomorphism
 $$\QQ\cong \QQ/(I\cap J)
\to \QQ/J\oplus \QQ/I \to \PP_1\oplus \PP_2=\PP''$$
as desired.

(2) This follows from part (1).
\end{proof}

For the rest of this section, we study some basic
results about saturation in the setting of graded algebras. These results will be useful in the sequel when we study saturated operads of GK-dimension 1. 

We refer to \cite[Definitions 3.5 and 4.6]{LQXZ25} for the
definitions of {\it graded Perm} (or {\it GPerm} for short) and
{\it pseudo-graded-Perm} (or {\it PGPerm} for short) unital
associative algebras. And in this article, we mainly use a
special class of PGPerm algebras defined as follows.

\begin{definition}[{\cite[Definition 5.4]{LQXZ25}}]
\label{zzdef-new-3.5}
An ${\mathbb N}$-graded associative algebra $A=
\oplus_{i=0}^{\infty} A_i$ is called a {\it pseudo-graded commutative algebra}
(or {\it PGC algebra} for short) if it satisfies the following
conditions:
\begin{enumerate}
\item[(i)]
For each $i\geq 1$, $A_i=A_{i,e}\oplus A_{i,o}$.
Elements in $A_{i,e}$ {\rm{(}}resp. $A_{i,o}$ or $A_0${\rm{)}}
are called {\it homogeneous of type} $(i,e)$ (resp. $(i,o)$ or $(0)${\rm{)}}.
Define
$$t(x)=\begin{cases} 0 & x\in A_{0}, \\
0 & x\in A_{i,e}, \\
1 & x\in A_{i,o}.
\end{cases}$$
\item[(ii)]
$A_{o}:=\oplus_{i\geq 1} A_{i,o}$ and
$A_{e}:=\oplus_{i\geq 1} A_{i,e}$ are two-sided ideals of $A$. As a consequence,  we have $A_{o}A_e=A_eA_o=0$.
\item[(iii)] For all elements $y,z\in A_0\cup A_{o}\cup A_{e}$,
\begin{equation}
\label{E3.5.1}\tag{E3.5.1}
y  z=(-1)^{\deg(y)\deg(z)t(y)t(z)} z  y.
\end{equation}
\end{enumerate}
\end{definition}

By the comment after \cite[Definition 5.4]{LQXZ25}, every PGC
algebra is a PGPerm algebra. A PGC algebra $A$ is said to be of
{\it even type} (resp. {\it odd type}) if $A_{o}=0$ (resp.
$A_{e}=0$). It is clear that a PGC algebra of even type is
a commutative ${\mathbb N}$-graded algebra. By \cite[Lemma 5.5]{LQXZ25}, a
PGC algebra of odd type is equivalent to a graded commutative
algebra since in this case \eqref{E3.5.1} becomes
$$yz=(-1)^{\deg(y)\deg(z)}zy.$$
By convention, all PGC algebras of odd type are assumed to be over a field of characteristic $\neq 2$.
Note that a PGC algebra homomorphism~$\varphi:A\longrightarrow B$ is an $\mathbb{N}$-graded algebra homomorphism such that~$\varphi(A_{i,o})\subseteq B_{i,o}$ and~$\varphi(A_{i,e})\subseteq B_{i,e}$ for all~$i\geq 1$. 

Now we construct a commutative ${\mathbb N}$-graded algebra, which will soon be proved to be saturated.
Recall that for every integer~$b\in \mathbb{N}$,  an arbitrary ${\mathbb Z}/(b+1)$-graded algebra~$B$
is called \emph{nilpotent} if $\oplus_{i=1}^b B_i$
is a nilpotent ideal of $B$.
\begin{example}
\label{zzex-new-3.6}
Let $(B, \star)$ be a  commutative ${\mathbb Z}/(b+1)$-graded algebra
with Hilbert series $1+t+\cdots+t^b$ for some integer
$b\geq 0$. We define a connected commutative ${\mathbb N}$-graded algebra
$(B\{c\}, \cdot)$ with $\deg c=b+1$ as follows.

For $0\leq i\leq
b$, we pick a nonzero basis element $x_i$ of $B_i$ with $x_0=1$
and defined~$B\{c\}_i=\Bbbk x_i$. For $i\geq b+1$, let
$\bar{i}$ denote an integer between $0$ and $b$ such that
$i=\bar{i} \mod (b+1)$. In this case let $q_i=\frac{i-\bar{i}}{b+1}$.
Define~$B\{c\}_i=\Bbbk x_{\bar{i}} c^{q_i}$.
By construction $\{x_i c^j\mid 0\leq i\leq b, j\geq 0\}$ is a $\Bbbk$-linear
basis of $B\{c\}$. The commutative associative multiplication
$\cdot$ of $B\{c\}$ is defined to be
\begin{equation}\label{ex3.6eq}\tag{E3.6.1}
 x_{i_1} c^{j_1} \cdot x_{i_2} c^{j_2}:=\begin{cases}
(x_{i_1}\star x_{i_2}) c^{j_1+j_2}, & i_1+i_2< b+1,\\
(x_{i_1}\star x_{i_2}) c^{j_1+j_2+1}, & i_1+i_2\geq b+1,
\end{cases}
\end{equation}
where~$0\leq i_1,i_2\leq b$.
When $B$ is nilpotent,  $c$ is a 
nonzerodivisor in
$B\{c\}$ of the smallest positive  degree. Note that~$B\{ c\}$ is a PGC algebra of even type and we frequently omit the multiplication symbol~$\cdot$.   
\end{example}

A graded commutative 
algebra can be constructed analogously.
 \begin{example}\label{zzrema-new-3.7}
 Let ~$\Bbbk$ be a field with ${\text{char}}\; \Bbbk\neq 2$. Let
 $(B, \star)$ be a (${\mathbb Z}/(b+1)$-graded) graded  commutative algebra over $\Bbbk$ with Hilbert series $1+t+\cdots+t^b$ for some odd integer~$b\geq 1$. Let $(B\{c\}, \cdot)$ be the algebra defined by~\eqref{ex3.6eq}. Then $(B\{c\}, \cdot)$ is a graded commutative algebra with $\deg c=b+1$. Note that $(B\{c\}, \cdot)$ is a PGC algebra of odd type. 
\end{example}

A left torsionfree PGC algebra $A$ is called a {\it saturated PGC algebra}, if
for every left torsionfree PGC algebra $B$ with a  PGC algebra isomorphism  
$\phi:B_{\{w\}}\to A_{\{w\}}$  
 for some integer~$w>0$, the map
$\phi$ extends to an injective
PGC algebra homomorphism $\phi': B \to A$.

\begin{lemma} \label{n-sat-pgc-sat}
   If a PGC algebra~$A$ is saturated as an $\mathbb{N}$-graded algebra, then~$A$ is a saturated PGC algebra.  
\end{lemma}
\begin{proof}
Let $\varphi_{\{w\}}:B_{\{w\}}\to A_{\{w\}}$  be a PGC algebra isomorphism. Then $\varphi_{\{w\}}$ can be extended to be an $\mathbb{N}$-graded algebra injective homomorphism~$\varphi: B\to A$.
For every~$b\in B_{i,o}$, assume~$\varphi(b)=a_o+a_e$ for some~$a_o\in A_{i,o}$ and~$a_e\in A_{i,e}$.  Since~$B_{p,e}b=0$ for all~$p\geq w$ and since~$\varphi_{\{w\}}$ is an isomorphism, it follows that~$A_{p,e}a_e=A_{p,e}(a_o+a_e)=\varphi(B_{p,e}b)=0$ and thus~$A_{p}a_e=(A_{p,o}+A_{p,e})a_e=0$ for every~$p\geq w$. Since~$A$ is a left torsionfree PGC algebra, we have~$a_e=0$, and thus $\varphi(B_{i,o})\subseteq A_{i,o}$. Similarly, we have $\varphi(B_{i,e})\subseteq A_{i,e}$. So $\varphi$ is a PGC algebra homomorphism.
\end{proof}

Let $x$ be a nonzero element in an algebra $A$. Then the left multiplication~$l_{x}$ is defined to be~$y\mapsto xy$ for every $y\in A$.

\begin{lemma}\label{saturated-condi}
Let $A$ be an ${\mathbb N}$-graded algebra. Suppose
that there is a sequence of nonzero homogeneous
elements $\{\alpha_s\}_{s\geq 1}$  in the center of $A$ with strictly increasing
degrees, satisfying the following
condition: 
\begin{enumerate}
\item[($\ast$)]
For every integer $d$, there is an integer $t_d > d$
such that, for every integer $s\geq t_d$,  the left
multiplication $l_{\alpha_s}$ induces a bijective
map $\oplus_{i=0}^{d}
A_i\to \oplus_{i=a_s}^{d+a_s} A_i$,  where
$a_s=\deg(\alpha_s)$.
\end{enumerate}
Then 
\begin{enumerate}
\item[(1)]
There exists a strictly increasing sequence of integers
$\{d_m\}_{m\geq 1}$ and a subsequence
$\{\alpha_{s_m}\}_{m\geq 1}$ of $\{\alpha_s\}_{s\geq 1}$, such that 
\begin{enumerate}
\item[($\ast1$)]
$d_1=1$, $a_{s_m}> d_{m}=d_{m-1}+a_{s_{m-1}}+1$,
for all $m\geq 2$, and
\item[($\ast2$)]
for every $t\geq m$, the left multiplication
 $l_{\alpha_{s_t}}$ induces a bijective map
$\oplus_{i=0}^{d_{m}}
A_i\to \oplus_{i=a_{s_t}}^{d_{m}+a_{s_t}} A_i$.
\end{enumerate}
Consequently, we have the following
{\it cancellation property}: If $\alpha_{s_t} f
=\alpha_{s_t} g$, with $\deg(f),\deg(g)\leq
d_{m}$ and $t\geq m$, then $f=g$.
\item[(2)]
$A$ is saturated.
\end{enumerate}
\end{lemma}
 
\begin{proof}
 (1) 
Let $d_1 = 1$ and $s_1 = t_{d_1}$. Then by Condition ($\ast$), 
 for every $s \geq s_1$, the left multiplication $l_{\alpha_s}$ 
induces a bijective map $\oplus_{i=0}^{d_{1}}
A_i\to \oplus_{i=a_s}^{d_{1}+a_s} A_i$. 
Now assume that there exist $d_1,d_2,\cdots,d_{m-1}$ and $\alpha_{s_1},\alpha_{s_2},\cdots,\alpha_{s_{m-1}}$
such that $a_{s_k}> d_{k}=d_{k-1}+a_{s_{k-1}}+1$ for all $2 \leq k \leq m-1$. Moreover, 
for each $1 \leq k \leq m-1$  and $s \geq s_k$, the left multiplication
 $l_{\alpha_{s}}$ induces a bijective map
$\oplus_{i=0}^{d_{k}}
A_i\to \oplus_{i=a_{s}}^{d_{k}+a_{s_k}} A_i$.
Define $d_m = d_{m-1} + a_{s_{m-1}} +1$.  
Then by Condition ($\ast$) again,  there exists an integer $t_{d_m} > d_m$ such that for every $s \geq t_{d_m}$, the left multiplication~$l_{\alpha_s}$ 
induces a bijective map $\oplus_{i=0}^{d_{m}}
A_i\to \oplus_{i=a_s}^{d_{m}+a_s} A_i$. In particular, we can set~$s_m=t_{d_m}$.  
Therefore, by induction on $m$, we obtain a strictly increasing sequence of integers
$\{d_m\}_{m\geq 1}$ and a subsequence
$\{\alpha_{s_m}\}_{m\geq 1}$ of $\{\alpha_s\}_{s\geq 1}$ such that ($\ast1$) and ($\ast2$) hold. The cancellation property follows immediately.

(2) 
Clearly, $A$ is left torsionfree.  Let $B$ be a left torsionfree
${\mathbb N}$-graded algebra such that $\phi:
B_{\{w\}} \to A_{\{w\}}$ is an isomorphism for
some positive integer $w$. It suffices to show
that $\phi$ extends to an injective graded algebra
homomorphism from~$B$ to $A$.

Retain the notation in (1). Then define
$\gamma_t=\alpha_{s_t}$  and $c_t=\deg(\gamma_t)$ for every~$t\geq 1$.   By deleting  the first
finitely many $\gamma_t$ if necessary, we may
assume that $c_1>w$.
Define $\beta_t=
\phi^{-1}(\gamma_t)$ for every~$t\geq 1$. Then $\deg(\beta_t)=\deg(\gamma_t)=c_t$. 

Suppose that $g\in \oplus_{i=0}^{d_m} B_i$ for some
$m$.  Note that for every $t\geq m$,  the left multiplication~$l_{\gamma_t}$ induces a bijective map
$\oplus_{i=0}^{d_{m}}
A_i\to \oplus_{i=c_{t}}^{d_{m}+c_{t}} A_i$ by ($\ast2$).
So for $\phi(\beta_t g)\in  \oplus_{i=c_{t}}^{d_{m}+c_{t}} A_i$,  there
is a unique element, denoted by
$\phi_{t}(g)$, in $\oplus_{i=0}^{d_m} A_i$
such that
\begin{equation}
\label{E1}\tag{E3.9.1}
\gamma_t \phi_t(g)=\phi(\beta_t g).
\end{equation}
\medskip
\noindent
{\bf Claim 1:}  If
$t,s\geq m$, then $\phi_{t}(g)=\phi_{s}(g)$. As
a consequence, $\phi_{t}(g)$ is independent
of the choice of $t$ as long as $t\geq m$.

\noindent
{\it Proof of Claim 1.}  
Let~$n\geq  m$ be an  integer satisfying~$d_{n}> c_t+c_s+d_m$, and let~$r\in \{t,s\}$. 
Since $\gamma_{n}$ and $\gamma_{r}$ lie in the center of~$A$ and $\phi$ is an isomorphism, we deduce that
$\beta_n$ commutes with $\beta_r$. So we obtain
$$\gamma_n (\gamma_r \phi_{r}(g))=\phi(\beta_n)\phi(\beta_r g)
=\phi(\beta_n \beta_r g)=\phi(\beta_r \beta_n g)\\
=\phi(\beta_r)\phi(\beta_n g)
= \gamma_r (\gamma_n\phi_n(g)),
$$
and thus~$\gamma_n \gamma_r \phi_{r}(g)=\gamma_n \gamma_r \phi_{n}(g)$. Note that~$d_n> c_r+d_m\geq \deg(\gamma_r)+\deg(g) $ and $r\geq m$. By applying the  
cancellation property repeatedly, we deduce that
$\phi_{r}(g)=\phi_{n}(g)$.
In particular, 
$\phi_{t}(g)=\phi_{n}(g)=\phi_{s}(g)$.
This
finishes the proof of  {\bf Claim 1}.

\medskip

It follows from {\bf Claim 1} that $\phi_{t}(g)$
is independent of the choice $t$ for $t\gg 0$.
Now we use $\phi_{\infty}(g)$ for $\phi_t(g)$.
When $\deg(g)\leq d_m$ and $t\geq m$, \eqref{E1}
becomes
\begin{equation}
\label{E2}\tag{E3.9.2}
\phi_{\infty}(g)=l_{\gamma_{t}}^{-1}(\phi(\beta_t g)).
\end{equation}
By ($\ast2$), the notation~$l_{\gamma_{t}}^{-1}$ makes sense. Since $l_{\gamma_{t}}^{-1}(-)$
and $\phi(\beta_t -)$ are $\Bbbk$-linear, one sees that

\medskip
\noindent
{\bf Claim 2:} $\phi_{\infty}$ is $\Bbbk$-linear.

\medskip
\noindent
{\bf Claim 3:} When restricted to $B_{\{w\}}$,
$\phi_{\infty}=\phi$.

\noindent
{\it Proof of Claim 3.} For $g\in B_{\{w\}}$ with degree $\leq d_m$ for
some $m$, choose an arbitrary integer $t\geq m$,
then $
\gamma_t \phi_{\infty}(g)
=\gamma_t \phi_{t}(g)=\phi (\beta_t g)
=\phi(\beta_t)\phi(g)=\gamma_t \phi(g)$.  
The claim follows from the cancellation property.

\medskip
\noindent
{\bf Claim 4:} For $g\in B$, there is an integer $t_g$ such that
$g\beta_{t}=\beta_t g$ for all $t\geq t_g$.

\noindent
{\it Proof of Claim 4.} Assume that $\deg(g)\leq d_m$ and set~$t_g=m$. Let $f:=\beta_t g-g \beta_t$ which is in
$B_{\{w\}}$ since $\deg (\beta_t)\geq w$ by assumption.
For $s\gg t\geq t_g$, we have
$$\begin{aligned}
\gamma_s \phi(f)
&=\phi(\beta_s) \phi(\beta_t g- g \beta_t)
=\phi(\beta_s \beta_t g)-\phi(\beta_s g \beta_t)\\
&=\phi(\beta_s) \phi(\beta_t g)-\phi(\beta_s g)\phi(\beta_t)
=\gamma_s \gamma_t \phi_{\infty}(g)-\gamma_{s} \phi_{\infty}(g)
\gamma_t\\
&=\gamma_s \gamma_t \phi_{\infty}(g)-\gamma_{s}\gamma_t \phi_{\infty}(g)
=0.
\end{aligned}
$$
Applying the cancellation property, we obtain
$\phi(f)=0$, which forces $f=0$.

\medskip
\noindent
{\bf Claim 5:} $\phi_{\infty}$ preserves multiplication.

\noindent
{\it Proof of Claim 5.} For all $f,g\in B$ of degree $\leq d_m$ for
some $m$, choose $s\gg t\gg m$, then we have
$$\begin{aligned}
\gamma_s \gamma_t \phi_{\infty} (fg)
&=\phi(\beta_s)\phi(\beta_t fg)\\
&=\phi(\beta_s \beta_t fg)=\phi(\beta_s f \beta_t g)
\qquad\qquad \qquad {\text{by {\bf Claim 4}}}\\
&=\phi(\beta_s f)\phi(\beta_t g)=\gamma_s \phi_{\infty}(f)
\gamma_t \phi_{\infty}(g)\\
&=\gamma_s \gamma_t \phi_{\infty}(f)\phi_{\infty}(g).
\end{aligned}
$$
By applying the cancellation property twice,
we deduce
$\phi_{\infty} (fg)=\phi_{\infty}(f)\phi_{\infty}(g)$.

\medskip
\noindent
{\bf Claim 6:} $\phi_{\infty}$ is injective.

\noindent
{\it Proof of Claim 6.} For $g\in B$ with
$\phi_{\infty}(g)=0$, by {\bf Claim 3} and {\bf Claim 5}, we have
$$
\phi(B_{\geq w} g)= \phi_{\infty}(B_{\geq w} g)
=\phi_{\infty}(B_{\geq w})\phi_{\infty}(g)
=0.
$$
Since $\phi$ is an isomorphism, we obtain $B_{\geq w}g=0$.  Since $B$ is left torsionfree, $g=0$.

The assertion follows by combining the above claims.
\end{proof}

\begin{corollary}
\label{coro3.12-saturated}
Let $A$ be an ${\mathbb N}$-graded algebra with Hilbert series $\frac{n}{1-t^r}$ for
some positive integers $r$ and $n$. If $A$ contains a homogeneous nonzerodivisor of positive degree in the center of~$A$,
then $A$ is saturated.
\end{corollary}

\begin{proof}
Let $x$ be a homogeneous nonzerodivisor of positive degree in the center of~$A$. 
By the assumption on the Hilbert series of~$A$, we have~$${\mathop\sum\limits_{i=0}^m}\dim(A_i) = {\mathop\sum\limits_{i=\deg(x^t)}^{m+\deg(x^t)}}\dim(A_i)$$ for all~$m \geq 0$ and $t\geq 1$.  
It follows that the left multiplication by~$x^t$ induces a bijective
map $\oplus_{i=0}^{m}
A_i\to \oplus_{i=\deg(x^t)}^{m+\deg(x^t)} A_i$.
Let~$\alpha_t=x^t$.  Then the sequence $\{\alpha_t\}_{t\geq 1}$ satisfies  
Condition ($\ast$) in Lemma \ref{saturated-condi}, and thus the assertion
follows.
\end{proof}

 As specific examples of Corollary~\ref{coro3.12-saturated}, we immediately have the following  lemmas.  

\begin{lemma}\label{zzlem-new-3.8}
 Let $B\{c\}$ be an ${\mathbb N}$-graded algebra defined as in Example~\ref{zzex-new-3.6} or in Example~\ref{zzrema-new-3.7}.  Then $B\{c\}$ is saturated as an $\mathbb{N}$-graded algebra.  Consequently, $B\{c\}$ is a saturated PGC algebra.  
\end{lemma}
 
 \begin{lemma}\label{zzlem-new3-3.9}
 Let~$F$ be a finite field extension of~$\Bbbk$ and let~$t$ be a variable of positive degree.  Then $F[t]$ is saturated either as an ${\mathbb N}$-graded $F$-algebra or $\Bbbk$-algebra. 
\end{lemma}
  
 The following lemma plays an important role in the proof of Theorem~\ref{zzthm0.7}.
\begin{lemma}\label{zzlem-new-3.10}
 Let $A$ be a finitely generated  torsionfree  PGC
algebra of even type (resp. odd type) with Hilbert series $\frac{1}{1-t}$. Then $A$ is
isomorphic to $B\{c\}$ given in Example \ref{zzex-new-3.6} (resp. Example \ref{zzrema-new-3.7}), where~$B$ is nilpotent. Consequently,  $A$ is saturated. 
\end{lemma}
\begin{proof}
  We first show that~$A$ has a regular element (i.e., nonzerodivisor)~$c$. In addition, we will show that~$\deg(c)$ is an even integer if~$A$ is of odd type and ${\text{char}}\; \Bbbk\neq 2$. Assume that~$A$ is generated by~$X:=\{x_i \mid 1\leq i\leq n\}$ with $\deg(x_i)=i$. Let~$X_1=\{x_i \mid 1\leq i\leq n, x_i^m=0 \mbox{ for some } m\geq 2\}$ and let~$X_2=X\setminus  X_1$.  Since $A$ is infinite dimensional,  we deduce that the set~$X_2$ is nonempty. For all~$x_i,x_j\in X_2$, there exists a nonzero scalar~$\alpha_{i,j}$ such that~$x_{j}^i=\alpha_{i,j}x_i^j$.

   We claim that every~$x_i\in X_2$ is regular. Suppose to the contrary that~$x_i\in X_2$ is a zero divisor and~$x_iu=0$ for some $0\neq u\in A$. Then we have~$x_j^iu=\alpha_{i,j}x_i^ju=0$ for every~$x_j\in X_2$. It follows that~$x_q^m u=0$ for all~$x_q\in X$ and~$m\gg 0$. So~$u$ is a (left=right) torsion element, which contradicts the assumption that~$A$ is torsionfree.

Among all homogeneous nonzerodivisors of~$A$,  let $c$ be of minimal degree, say $b+1$.
If~$A$ is of odd type, then by convention, we have~${\text{char}}\; \Bbbk\neq 2$.  In this case, $\deg(c)$ is a positive even integer because~$c^2\neq 0$.

Let $B=A/(c-1)$. Then~$B$ is a ${\mathbb Z}/(b+1)$-graded algebra with Hilbert series $1+t+\cdots +t^b$.
Moreover, ~$B$ is commutative when $A$ is a PGC algebra of even type,
and~$B$ is graded commutative when $A$ is a PGC algebra of odd type
(in this case~$B$ is ${\mathbb Z}/(b+1)$-graded and $b+1$ is even).
Now we claim that $B$ is nilpotent. For every~$x_i\in A_i$ with
$1\leq i\leq b$, we deduce that $x_i^{b+1}$ is either 0 or a nonzero scalar
multiple of $c^i$. The latter case cannot happen since $c$ is a
nonzero divisor of the smallest positive degree. So~$x_i^{b+1}=0$ and $B$ is
nilpotent. It is easy to see that~$A$ is isomorphic to $B\{c\}$. The consequence follows from Lemma~\ref{zzlem-new-3.8}.
\end{proof}

\section{Equivalence of categories based on \cite{LQXZ25}}
\label{zzsec-cate-equi}
In this section, we continue to study the equivalence of categories based on \cite{LQXZ25}. The advantage of applying the equivalence lies in the fact that it is in general more convenient to consider the corresponding graded associative algebras. We first recall some notions from~\cite{LQXZ25}.

\begin{definition}[\cite{LQXZ25}]
\label{zzdef4.1}
Let $\PP$ be an operad.
\begin{enumerate}
\item[(1)]
An element $\lambda\in \PP(n)$ is called {\it
$\SG$-trivial} {\rm{(}}or {\it $\SG\triv$}{\rm{)}}
if $\lambda\ast \sigma=\lambda$ for all $\sigma
\in \SG_n$.
\item[(2)]
We say $\PP$ is {\it $\SG$-trivial} {\rm{(}}or
{\it $\SG\triv$}{\rm{)}} if every element in $\PP(n)$
is $\SG\triv$ for all $n$.
\item[(3)]
We say $\PP$ is {\it almost $\SG$-trivial} {\rm{(}}or
{\it almost $\SG\triv$}{\rm{)}} if every element in $\PP(n)$
is $\SG\triv$ for all $n\gg 0$.
\item[(4)]
An element $\lambda$ in $\PP(n)$ is called {\it $\AG$-trivial}
(or {\it $\AG\triv$}) if $\lambda\ast \sigma=\lambda$
for all $\sigma\in \AG_n$.
\item[(5)]
$\PP$ is called {\it $\AG$-trivial} (or {\it $\AG\triv$})
{\rm{(}}resp. {\it almost $\AG$-trivial} (or
{\it almost $\AG\triv$}{\rm{)}}) if every element in $\PP(n)$ is $\AG\triv$
for all $n$ {\rm{(}}resp. for all $n\gg 0${\rm{)}}.
\item[(6)]
An element $\lambda$ in $\PP(n)$ is called \emph{$\SG$-signed} (or {\it $\SG\sign$})
if $\lambda\ast \sigma=\sgn(\sigma)\lambda$ for all
$\sigma\in \SG_n$.
\item[(7)]
$\PP$ is called {\it $\SG\sign$} {\rm{(}}resp.
{\it almost $\SG\sign$}{\rm{)}} if every element
in $\PP(n)$ is $\SG\sign$ for all $n$ {\rm{(}}resp. for all
$n\gg0${\rm{)}}.
\end{enumerate}
\end{definition}

We can define a functor $G_{\SG\triv}$ from the category of GPerm
algebras to the category of $\SG$-trivial operads and a functor~${\mathcal F}_{\SG\triv}$ from the category of~$\SG\triv$ operads to the category of GPerm algebras, see  \cite[Lemma 3.7, Remark 3.8]{LQXZ25}.  Similarly, assuming~${\text{char}}\; \Bbbk\neq 2$, then for every PGPerm algebra $(A,\cdot)$, we can construct an~$\AG\triv$ operad~$\PP$ as in~\cite[Lemma 4.9]{LQXZ25}, which is denoted by~$G_{\AG\triv}(A)$.  Since~$\PP$ is~$\AG\triv$, the composition $x\circ_i y $, for $1\leq i\leq \Ar(x)$,  is determined by $x\circ_1 y$ for all~$x,y\in\PP$. Moreover, we have~$x\circ_1 y=x\cdot y$.   On the other hand, given an~$\AG\triv$ operad~$\PP$, we can construct a PGPerm algebra~${\mathcal F}_{\AG\triv}(\PP)$ from~$\PP$, see in \cite[Lemma 4.7]{LQXZ25}.  By~\cite[Theorems 3.9 and 4.11]{LQXZ25}, we have the following:

 (i) the functors $(G_{\SG\triv}, {\mathcal F}_{\SG\triv})$ induce
an equivalence between the category of GPerm algebras and the
category of $\SG\triv$ operads;

(ii) the functors $(G_{\AG\triv}, {\mathcal F}_{\AG\triv})$ induce
an equivalence between the category of PGPerm algebras and the
category of $\AG\triv$ operads.

The following is based on \cite[Corollary 3.10]{LQXZ25}:

\begin{theorem}
\label{zzthm-new-4.2}
The functors $(G_{\SG\triv}, {\mathcal F}_{\SG\triv})$ restrict to
\begin{enumerate}
\item[(i)]
an equivalence between the category of saturated 
commutative
${\mathbb N}$-graded algebras and that of $\SG\triv$ saturated operads, and
\item[(ii)]
an equivalence between the category of  infinite dimensional prime saturated commutative
${\mathbb N}$-graded algebras and that of $\SG\triv$ infinite dimensional  prime saturated operads.
\end{enumerate}
\end{theorem}

\begin{proof}
(i) Suppose that $A$ is saturated and commutative, and let $\PP=G_{\SG\triv}(A)$.
By definition, $A$ is (left=right) torsionfree. By
\cite[Corollary 3.10(1)]{LQXZ25}, $\PP$ is left torsionfree. Now let $\QQ$
be another left torsionfree operad
with $\QQ_{\{w\}}\cong \PP_{\{w\}}$. Then $\QQ$ is almost $\SG\triv$.
 By \cite[Lemma 3.13]{LQXZ25}, $\QQ$ is $\SG\triv$. By
\cite[Corollary 3.10(1)]{LQXZ25}, $B:={\mathcal F}_{\SG\triv}(\QQ)$
is a torsionfree commutative algebra. The equation
$\QQ_{\{w\}}\cong \PP_{\{w\}}$ implies that $B_{\{w\}}
\cong A_{\{w\}}$. Since $A$ is saturated, there is an injective map
$\phi: B\to A$ extending the isomorphism $B_{\{w\}}\cong A_{\{w\}}$.
By \cite[Corollary 3.10(1)]{LQXZ25}, $G_{\SG\triv}(\phi): \QQ\to \PP$
extends the isomorphism $\QQ_{\{w\}}\cong \PP_{\{w\}}$. Thus $\PP$
is saturated.

The converse can be proven in a similar way together with the following fact:
Let $A$ be a saturated commutative algebra and $B$ an arbitrary left torsionfree algebra
such that $B_{\{w\}}\cong A_{\{w\}}$ for some $w\ge 0$. Then $B$ is commutative.

(ii) By
\cite[Corollary 3.10(2)]{LQXZ25} and part (1), the assertion follows.
\end{proof}

\begin{corollary}\label{zzcor-new-4.3}
Let $A$ be $B\{c\}$ as in Example~\ref{zzex-new-3.6} such that~$B$ is nilpotent and let $\Lin_e(B)$ be the operad $G_{\SG\triv}(A)$.
Then we have the following:
\begin{enumerate}
\item[(1)]
The Hilbert series of $A$ is $\frac{1}{1-t}$ and
the Hilbert series of $\Lin_e(B)$ is $\frac{t}{1-t}$.
\item[(2)]
$A$ is a torsionfree and saturated graded algebra and
$\Lin_e(B)$ is a finitely generated saturated
$\SG\triv$ operad.
\item[(3)]
 $A$ is prime if and only if $b=0$. 
\item[(4)]
$\Lin_e(B)$ is prime if and only if $b=0$.
\end{enumerate}
\end{corollary}

\begin{proof}
  Part (2) follows from Lemma~\ref{zzlem-new-3.8},  Theorem~\ref{zzthm-new-4.2}(i) and~\cite[Corollary 3.10(1)]{LQXZ25}.  What remains is straightforward, and thus the proofs are omited.
\end{proof}

\begin{corollary}
Every finitely generated saturated $\SG\triv$
operad with Hilbert series $\frac{t}{1-t}$ is of the form $\Lin_e(B)$.
\end{corollary}

 \begin{proof} Note that commutative
${\mathbb N}$-graded algebras can be viewed as PGC algebras of even type.  The result follows from Theorem
\ref{zzthm-new-4.2}(i), and Lemma~\ref{zzlem-new-3.10}.
\end{proof}

Now we consider the saturated version of \cite[Corollary 5.7]{LQXZ25}.
\begin{theorem}
\label{zzthm-new-4.5}
Suppose ${\text{char}}\; \Bbbk\neq 2$. The functors
$(G_{\AG\triv}, {\mathcal F}_{\AG\triv})$ restrict to
\begin{enumerate}
\item[(i)]
an equivalence between the category of saturated PGC algebras and
that of $\AG\triv$ saturated operads, and
\item[(ii)]
an equivalence between the category of 
infinite dimensional prime saturated PGC algebras
and that of $\AG\triv$ infinite dimensional  prime saturated operads. 
\end{enumerate}
\end{theorem}
\begin{proof}
(i) Suppose $A$ is a saturated PGC algebra. By definition, $A$ is
(left and right) torsionfree. Let $\PP=G_{\AG\triv}(A)$.
By  \cite[Corollary 5.7(1)]{LQXZ25},  $\PP$ is left torsionfree. Now
let $\QQ$ be another left torsionfree operad with $\QQ_{\{w\}}\cong
\PP_{\{w\}}$. Then $\QQ$ is almost $\AG\triv$. By
\cite[Lemma 3.13(2)]{LQXZ25}, $\QQ$ is $\AG\triv$. By
\cite[Corollary 5.7(1)]{LQXZ25},
$B:={\mathcal F}_{\AG\triv}(\QQ)$ is a torsionfree PGC algebra.
The equation $\QQ_{\{w\}}\cong \PP_{\{w\}}$ implies that $B_{\{w\}}
\cong A_{\{w\}}$. Since $A$ is saturated, there is an injective homomorphism
$\phi: B\to A$ extending the isomorphism $B_{\{w\}}\cong A_{\{w\}}$.
By \cite[Corollary 5.7(1)]{LQXZ25},  $G_{\AG\triv}(\phi): \QQ\to \PP$
extends the isomorphism $\QQ_{\{w\}}\cong \PP_{\{w\}}$. Thus $\PP$
is saturated.

By a similar proof as above, the converse of the statement follows. 

(ii
) By
\cite[Corollary 5.7(2)]{LQXZ25} and part (1), the result follows.
\end{proof}

\begin{remark}\label{zzrema-new3-4.6}
 Suppose ${\text{char}}\; \Bbbk\neq 2$.   Let $A$ be the algebra $B\{c\}$ given in Example~\ref{zzrema-new-3.7} such that~$B$ is nilpotent and let $\Lin_{o}(B)$ be the operad $G_{\AG\triv}(A)$.  Then $\Lin_{o}(B)$ is $\SG\sign$, and it is not prime if~$B_{\geq 1}\neq \{0\}$.  Clearly, the Hilbert series of $\Lin_{o}(B)$ is $\frac{t}{1-t}$.
Finally, by Theorem~\ref{zzthm-new-4.5}(i), $\Lin_{o}(B)$ is a finitely generated  saturated linear operad. 
\end{remark}

In view of Theorems~\ref{zzthm-new-4.2} and \ref{zzthm-new-4.5},
we shall prove some basic facts about commutative graded algebras and
PGC algebras. These facts will be useful in the next section. We begin with the following observation.

\begin{lemma}
\label{lem-quo-ring}
Let $A$ be a commutative graded algebra 
of GK-dimension~1. Let~$x$ be a 
nonzero homogeneous element of~$A$ of 
positive degree and let~$B=A[x^{-1}]$. 
Suppose that~$A$ is a domain. Then the 
following hold. 
\begin{enumerate}
\item[(1)]  
$B$ is the graded quotient ring 
of~$A$.
\item[(2)] 
$B=F[t,t^{-1}]$ and~$A\subseteq 
F[t]$, where~$F:=B_{0}$ is 
an algebraic field extension of~$\Bbbk$ 
{\rm{(}}of GK-dimension zero{\rm{)}}, 
and~$t$ is an arbitrary nonzero 
homogeneous element in $B$ of minimal 
positive degree.
\item[(3)] 
If~$A$ is finitely generated, then there 
exists a positive integer~$m>\deg(x)$ 
such that for every $n\geq m$, we 
have~$A_n=xA_{n-\deg(x)}=(A[x^{-1}])_n
=(F[t])_n$.
\end{enumerate}
\end{lemma}

\begin{proof}
(1) Let~$\Omega=\{x^n\mid n\in \mathbb{N}\}$ such that~$x^0=1$. Then~$\Omega$ is an Ore set of~$A$ and we have~$B=A[x^{-1}]=A\Omega^{-1}$. For every homogeneous element~$a\in A$, we have
$$1= \GKdim(\Bbbk[x])\leq \GKdim(\Bbbk[a,x])\leq 
\GKdim A=1,$$ 
thus~$a$ is algebraic over~$\Bbbk[x]$. 
Then there exist homogeneous elements~$g_0,
\dots, g_n \in \Bbbk[x]$ such that~$g_na^n
+\cdots +g_1a+g_0=0$. Since $a$ is 
homogeneous, we can choose $g_i$ to 
be homogeneous or $g_i=c_i x^{n_i}$
for some $c_i\in \Bbbk$ and $n_i\in 
{\mathbb N}$. Since~$a$ is not a 
zero-divisor, we may further assume that~$g_0\neq 0$. So $a^{-1}\subseteq \Bbbk[a,x,x^{-1}]\subseteq A[x^{-1}]$.
Therefore, for all homogeneous 
elements~$0\neq a,b\in A$, we have~$ba^{-1}
\in A[x^{-1}]$. It follows that~$A[x^{-1}]$ 
is the graded quotient ring of~$A$, and 
every nonzero homogeneous element is 
invertible.

(2) By \cite[Proposition 4.2]{KL00}, we 
have $\GKdim (B)=\GKdim (A)=1$.  Since $\Bbbk[x]\subseteq B_{0}[x]\subset B$, 
it follows that $\GKdim(B_{0}[x])=1$ and 
thus $\GKdim(B_{0})=0$. So $B_0$ is an 
algebraic field extension of~$\Bbbk$. 

Now we show that~$B=B_{0}[t,t^{-1}]$. For 
every homogeneous element~$b\in B$ of 
positive degree, assume that~$\deg(b)=p\deg(t)+q$, 
where~$0\leq q<\deg(t)$. Then 
$0\leq\deg(bt^{-p})=q<\deg(t)$. By the 
choice of~$t$, we obtain~$bt^{-p}\in 
B_0$. It follows immediately that~$b
\in B_{0}[t]$ and thus~$B\subseteq B_{0}[t,t^{-1}]\subseteq B$. Clearly, 
$A\subseteq B_0[t]$.

(3) Since~$\GKdim(A/xA)=0$ and~$A$ is 
finitely generated,  $A/xA$ is finite dimensional. So there exist an integer~$m>\deg(x)$ such that~$A_n
=(xA)_n=xA_{n-\deg(x)}$ for every $n\geq m$. 
So for every~$bx^{-p}\in (A[x^{-1}])_n$ with~$b\in A$, we have~$b\in A_{n+p\deg(x)}
=x^pA_n$, and thus $bx^{-p}\in 
x^p A_n x^{-p}=A_n$. It follows that
$$(F[t])_n=(F[t,t^{-1}])_n=(A[x^{-1}])_n\subseteq A_n\subseteq (A[x^{-1}])_n\subseteq (F[t])_{n}.$$
The proof is completed.
\end{proof}

\begin{proposition}
\label{zzpro-new3-4.7}
Let $A$ be a prime commutative graded algebra of GK-dimension~1. Then we have the following:
\begin{enumerate}
\item[(1)]
$A$ is a subalgebra of $\overline{\Bbbk}[t]$  with $\deg t>0$. 
\item[(2)]
The following are
equivalent:
\begin{enumerate}
\item[(2a)]
$A$ is finitely generated.
\item[(2b)]
 $A$ is a $\Bbbk$-subalgebra of $F[t]$ with $\deg t>0$, where $F$ is a finite
field extension of $\Bbbk$.
\item[(2c)]
There is an integer $d$ such that $\dim  A_n \leq d$ for all
$n$.
\item[(2d)]
The Hilbert series of $A$ is rational.
\end{enumerate}
\item[(3)]
Suppose $\Bbbk$ is algebraically closed. Then $A$ is a subalgebra of $\Bbbk[t]$. As a
consequence, it is connected (i.e., $A_0=\Bbbk$)   and finitely generated.
\end{enumerate}
\end{proposition}

\begin{proof} Every commutative prime algebra is a domain. Since we implicitly assume that $A$ is locally finite, $A\neq A_0$ or there is a nonzero element $x$ of positive degree. So we can apply Lemma~\ref{lem-quo-ring}.

(1) By Lemma~\ref{lem-quo-ring}(2) and the notation introduced there, $A$ is a subalgebra of $F[t^{\pm 1}]$, where $F$ is field of GK-dimension zero and $\deg t>0$. Since $F$ is a
subfield of $\overline{\Bbbk}$, the claim follows.

(2) (2a) $\Rightarrow$ (2d):
This follows by a similar reasoning of \cite[Proposition 3.15(2)]{LQXZ25}.

(2d) $\Rightarrow$ (2c): Write $H_{A}(t)=\frac{f(t)}{g(t)}$.
Since $\GKdim (A)=1$, we have $f(1)\neq 0$ and
$g(t)=(1-t)h(t)$ with $h(1)\neq 0$. By \cite[Theorem 4.1]{Sta78}
or the proof of \cite[Corollary 2.2]{SZ97}, $\dim A_n$ for
$n\gg 0$ is a multi-polynomial function of degree 0 in the sense
of \cite[p.1597]{SZ97}. Therefore $\dim A_n$ is uniformly
bounded.

(2c) $\Rightarrow$ (2b):
By Lemma~\ref{lem-quo-ring}(2,3), 
$A$ is a $\Bbbk$-subalgebra of~$F[t^{\pm 1}]$ and~$(F[t])_n=A_n$ for $n\gg 0$, where~$F$ is an algebraic field extension of~$\Bbbk$. Now we have, for $n\gg 0$,
$$\dim(F)=\dim(F t^n)=\dim A_{n \deg t}<\infty$$     
as $A$ is a locally finite.

(2b) $\Rightarrow$ (2a):
Let $F$ be a finite field extension of $\Bbbk$. Let $y:=f t^m$, with $f\in F$ and $m\geq 1$, be a nonzero homogeneous element in $A$ and let $Y:=\Bbbk[y]$ be the $\Bbbk$-subalgebra of $A$ generated by $y$. Then $F[t]$ is a finitely generated module over $Y$.
(In fact, $F[t]$ is a free $Y$-module
of rank $m\dim F$.) Since $Y\subseteq A$, $F[t]$ is a finitely
generated module over $A$. By Artin-Tate lemma, $A$ is finitely generated 
as a $\Bbbk$-algebra.

(3) By Lemma~\ref{lem-quo-ring}, it follows that~$A$ is a subalgebra of~$\Bbbk[t]$ because~$\Bbbk$ is algebraically closed. But every subring of $\Bbbk[t]$ is connected and
finitely generated by part (2).
\end{proof}

The following result is a PGC version of Proposition \ref{zzpro-new3-4.7}.

\begin{proposition}
\label{zzpro-new3-4.8}
Suppose ${\text{char}}\; \Bbbk\neq 2$. Let $A$ be a prime PGC
algebra. Then $A$ is of either even or odd type
and $A$ is commutative graded. If, in addition,  $A$ is  of GK-dimension 1, then  the following hold.
\begin{enumerate}
\item[(1)]
$A$ is a $\Bbbk$-subalgebra of $\overline{\Bbbk}[t]$ with $\deg t>0$.
\item[(2)]
The following are equivalent:
\begin{enumerate}
\item[(2a)]
$A$ is finitely generated.
\item[(2b)]
$A$ is a $\Bbbk$-subalgebra of $F[t]$ where $F$ is a finite field extension
of $\Bbbk$.
\item[(2c)]
There is an integer $d$ such that $\dim  A_n \leq d$ for all
$n$.
\item[(2d)]
The Hilbert series of $A$ is rational.
\end{enumerate}
\item[(3)]
Suppose $\Bbbk$ is algebraically closed. Then $A$ is a subalgebra
of $\Bbbk[t]$. As a consequence, it is connected  and finitely
generated.
\end{enumerate}
\end{proposition}

\begin{proof}
By Proposition \ref{zzpro-new3-4.7}, 
it suffices to show that $A$ is of either even or odd type
and $A$ is commutative graded.
 By definition, $I:=\oplus_{i\geq 1} A_{i,e}$ and
$J:=\oplus_{i\geq 1} A_{i,o}$ are ideals such that
$IJ=0$. Since $A$ is prime, either $I=0$ or $J=0$, namely, $A$ is of either even or odd type.
If $J=0$, then $A$ is a commutative graded algebra.  If~$I=0$, then~$A$ is a PGC algebra of odd type. By~\eqref{E3.5.1},
$A$ is a graded commutative algebra, and
for every element $x$ of odd degree, we have $x^2=0$.  Since $A$ is prime, $x=0$ for
all $x\in A_{i,o}$ when $i$ is odd, whence $A$ is commutative graded. The proof is completed.
\end{proof}

 In light of Propositions~\ref{zzpro-new3-4.7} 
and~\ref{zzpro-new3-4.8}, we describe $G_{\AG\triv}(F[t])$ below. Recall that~$\Com_{F}^{\{w\}}$ is defined as a suboperad of $\Com_{F}$ in~\eqref{E0.3.1}, and~$\Ope_{F}^{\{2w\}}$ is defined as a suboperad of $\Ope_{F}$ in~\eqref{E0.3.2}.

\begin{lemma}\label{zzlem-new3-4.9}
Let $F$ be an extension field of $\Bbbk$ and  let $A$ be a prime PGC algebra over a field $F$. Suppose that, as graded algebras, $A$ is isomorphic to $F[t]$ with $0<\deg(t)=:w\in \mathbb{N}$. Then we have $G_{\SG\triv}(A)=\Com_{F}^{\{w\}}$ if ${\text{char}}\; F =2$ or~$A$ is of even type; and $G_{\AG\triv}(A)=\Ope_{F}^{\{w\}}$ with $w=2m$ for some positive integer~$m$  if~${\text{char}}\; F \neq 2$ and~$A$ is of odd type. 
\end{lemma}

\begin{proof}
By Proposition~\ref{zzpro-new3-4.8}, $A(:=F[t])$ is of either even type or odd type.

If~${\text{char}}\; F =2$  or $A$ is of even type, then by the construction  in \cite[Lemma 3.7]{LQXZ25}, we obtain~$F=P(1)$, where~$\PP=G_{\SG\triv}(A)$. Moreover, for all~$f,f_1,f_2\in F$ and~$\sigma\in \SG_{wp+1}$, we have $ft^p \ast \sigma =ft^p$, and $f_1t^p\circ_{i} f_2t^{q}=f_1f_2t^{p+q}$ for every~$1\leq i\leq wp+1$. 
  It follows that $G_{\SG\triv}(A)=\Com_{F}^{\{w\}}$.

  If~${\text{char}}\; F \neq 2$ and~$A$ is of odd type,   then by the construction in \cite[Lemma 4.9]{LQXZ25}, we have
  $G_{\AG\triv}(A)(wp+1)=Ft^p$ and $G_{\AG\triv}(A)(n)=0$ if~$n\neq wp+1$ for every  integer $p$. By the proof of  Proposition~\ref{zzpro-new3-4.8}, we obtain that~$A_{i,o}=0$ if $i$ is an odd integer,  and consequently, $\deg(t)=2m$ 
  where $m$ is a positive  integer. Thus $\Ar(t)=2m+1\geq 3$. In particular, for every~$q\geq 1$, we have $\Ar(t^q) \in 2\mathbb{N}+1$.
  By the construction in~\cite[Lemma 4.9]{LQXZ25},  we obtain~$f_1t^p\circ_{i} f_2t^{q}=f_1f_2t^{p+q}$ and~$ft^p \ast \sigma =\sgn(\sigma)ft^p$ for all~$1\leq i\leq wp+1$
 and~$\sigma\in \SG_{wp+1}$, where~$f,f_1,f_2\in F$. Therefore, we have~$G_{\AG\triv}(A)=\Ope_{F}^{\{w\}}$ with~$w=2m$. The proof is completed.
\end{proof}

Then we have the following easy observation on prime almost $\AG\triv$ operads.

\begin{lemma}\label{zzlem-new3-4.10}
Let~$\PP$ be an
infinite dimensional prime almost $\AG\triv$ operad. Then~$\PP$  is left torsionfree and $\AG\triv$. In particular, every finitely generated prime operad of GK-dimension 1 is left torsionfree and $\AG\triv$.  
\end{lemma}
\begin{proof}
  By \cite[Lemma 2.5(1)]{LQXZ25},
$\PP$ is left torsionfree.  By \cite[Lemma 3.13(2)]{LQXZ25}, $\PP$ is  $\AG\triv$.
It follows from Theorem \ref{zzthm-new-2.4}(2) that
a finitely generated operad of GK-dimension $1$ is almost $\AG\triv$ and thus the last assertion follows.
\end{proof}

\begin{proposition}
\label{zzpro-new3-4.11}
Suppose ${\text{char}}\; \Bbbk\neq 2$.
Let $\PP$ be a   finitely generated   prime almost
$\AG\triv$ operad of GK-dimension 1.
Then $\PP$ is a suboperad of $\Com_F$ or
$\Ope_F$, where $F$ is a finite field extension of $\Bbbk$.
\end{proposition}

\begin{proof}   By Lemma~\ref{zzlem-new3-4.10},
$\PP$ is left torsionfree and $\AG\triv$.   By  \cite[Corollary 5.7]{LQXZ25},
$\PP=G_{\AG\triv}(A)$ for a
prime PGC algebra $A$.  By Proposition \ref{zzpro-new3-4.8}, $A$ is  of either even or odd type, and $A$ is a subalgebra of $F[t]$ where $F$ is a finite field extension
of~$\Bbbk$.
If~$A$ is of even type, then~$\PP=G_{\AG\triv}(A)=G_{\SG\triv}(A)$ is a suboperad of $G_{\SG\triv}(F[t])$. By Lemma~\ref{zzlem-new3-4.9}, we have~$G_{\SG\triv}(F[t])=\Com_{F}^{\{w\}}$, which is a suboperad of
 $\Com_F$.
 Similarly, if~$A$ is of odd type,  then~$\PP$ is a suboperad of $\Ope_{F}^{\{2m\}}\subseteq \Ope_{F}$ for some positive integer~$m$. The proof is completed. 
\end{proof}

\section{Proofs of the classification results}
\label{zzsec5}

The goal of this section is to prove Theorems \ref{zzthm0.3},
\ref{zzthm0.4} and \ref{zzthm0.7}.

\subsection{Classification of prime almost $\AG\triv$
operads of GK-dimension 1}
\label{zzsec5.1}

In this subsection we continue to study prime almost $\AG\triv$ operad that may not be finitely generated.
\begin{theorem}
\label{zzthm5.1}
Let $\PP$ be a prime almost $\AG\triv$ operad of GK-dimension 1. Then the following statements hold.
\begin{enumerate}
\item[(1)]
$\PP$ is a suboperad of $\Com_{\overline{\Bbbk}}$ or
$\Ope_{\overline{\Bbbk}}$.
\item[(2)]
The following are equivalent:
\begin{enumerate}
\item[(2a)]
$\PP$ is finitely generated.
\item[(2b)]
$\PP$ is a $\Bbbk$-suboperad of $\Com_{F}$ or
$\Ope_{F}$, where $F$ is a finite field extension
of $\Bbbk$.
\item[(2c)]
There is an integer $d$ such that $\dim
\PP(n)\leq d$ for all $n$.
\item[(2d)]
The Hilbert series of $\PP$ is rational.
\end{enumerate}
\item[(3)]
If $\Bbbk$ is algebraically closed, then $\PP$ is
finitely generated and connected.
\end{enumerate}
\end{theorem}

\begin{proof} Since $\PP$ is locally finite as always and of GK-dimension 1, it is infinite dimensional.
By Lemma~\ref{zzlem-new3-4.10}, $\PP$ is left torsionfree and $\AG\triv$.
If ${\text{char}}\; \Bbbk\neq 2$, then by \cite[Corollary 5.7]{LQXZ25}, we have
$\PP=G_{\AG\triv}(A)$ for some prime PGC algebra~$A$ of GK-dimension 1.
The statements follow from Proposition \ref{zzpro-new3-4.8}, Proposition~\ref{zzpro-new3-4.11} and \cite[Corollary 5.7]{LQXZ25}.
If ${\text{char}}\; \Bbbk=2$, by \cite[Lemma 5.3]{LQXZ25}, $\PP$ is $\SG\triv$.  By  \cite[Corollary 3.10(2)]{LQXZ25},  we have
$\PP=G_{\SG\triv}(A)$ for some prime commutative graded algebra of GK-dimension 1. The assertion follows
from Proposition \ref{zzpro-new3-4.7} and \cite[Corollary 3.10(2)]{LQXZ25}. 
\end{proof}

When $\overline{\Bbbk}$ is infinite dimensional over
$\Bbbk$, there are infinitely generated prime operads
of GK-dimension 1, see Example \ref{zzex6.1}.

\begin{proof}[Proof of Theorem \ref{zzthm0.3}]
By Theorem \ref{zzthm-new-2.4}(2), $\PP$ is almost
$\AG\triv$. The assertion follows from Theorem \ref{zzthm5.1}(1).
\end{proof}

Here is a variation of Theorem \ref{zzthm0.3}.

\begin{theorem}
\label{zzthm-new3-5.2}
Suppose $\PP$ is  an infinite dimensional prime operad such that $\dim \PP(n)$ is uniformly
bounded. Then the following hold.
\begin{enumerate}
\item[(1)]
$\PP$ is finitely generated and every element in $\PP$ is
central.
\item[(2)]
If further $\Bbbk$ is algebraically closed, then $\PP$ is a
suboperad of $\Com$ or $\Ope$. Consequently, $\PP$ is
connected.
\end{enumerate}
\end{theorem}

\begin{proof}
(1) By Proposition \ref{zzpro1.4}, $\PP$ is
almost $\AG\triv$. By Theorem \ref{zzthm5.1}(2),
$\PP$ is finitely generated and a suboperad of
$\Com_{F}$ or $\Ope_F$ for a finite field extension
$F/\Bbbk$. It is clear that every element in
either $\Com_{F}$ or $\Ope_F$ is central.

(2) This follows from part (1) and Theorem \ref{zzthm5.1}(1).
\end{proof}

\begin{corollary}
\label{zzcor-new3-5.3}
(1) Let $F$ be a finite field extension of~$\Bbbk$. Then operads of the form $\Com^{\{w\}}_{F}$ or
$\Ope^{\{2w\}}_F$ 
are saturated (both as $F$-operads and $\Bbbk$-operads).

(2) The saturation of a finitely generated prime operad
of GK-dimension 1 is either $\Com^{\{w\}}_{F}$ or
$\Ope^{\{2w\}}_{F}$ for some finite field extension $F$ of $\Bbbk$
and $w\geq 1$. 
\end{corollary}

\begin{proof}
(1) By Lemma \ref{zzlem-new3-3.9}, the prime PGC algebra $F[t] $ of even or odd type is saturated.  Then by Theorem~\ref{zzthm-new-4.2} (if $F[t]$ is of even type or ${\text{char}}\; \Bbbk=2$) or by~Theorem~\ref{zzthm-new-4.5} (if $A$ is of odd type and ${\text{char}}\; \Bbbk\neq 2$), we deduce that~$\Com^{\{w\}}_{F}=G_{\SG\triv}(F[t])$ and~$\Ope^{\{2w\}}_F=G_{\AG\triv}(F[t])$ are saturated (both as an $F$-operad and a $\Bbbk$-operad).

 (2) Let $\PP$ be a finitely generated prime operad of
GK-dimension 1. Then by Theorem~\ref{zzthm-new-2.4}(2) and Lemma~\ref{zzlem-new3-4.10},
$\PP$  is left torsionfree and $\AG\triv$. By the proof of Theorem~\ref{zzthm5.1}, we have $\PP=G_{\SG\triv}(A)$ for a prime commutative graded
algebra $A$ if~${\text{char}}\; \Bbbk=2$;  and  $\PP=G_{\AG\triv}(A)$ for a prime PGC algebra $A$  of odd or even type such that $A$ is also commutative if ${\text{char}}\; \Bbbk\neq 2$; Moreover, in both cases, $A$ is a finitely generated graded $\Bbbk$-subalgebra of~$F[t]$ for some finite field extension~$F$ of~$\Bbbk$. It follows that~$A$ is a domain of GK-dimension~1, and thus by Lemma~\ref{lem-quo-ring}, we have~$A_n=(F[t])_n$ for~$n\gg 0$. By Lemma~\ref{zzlem-new3-3.9}, it follows that $F[t]$ is the
saturation of $A$. By \cite[Theorems 3.9, 4.11]{LQXZ25},  $\PP$ is a suboperad of~$\QQ:=\Com^{\{w\}}_{F}=G_{\SG\triv}(F[t])$ (or~$\QQ:=\Ope^{\{2w\}}_{F}=G_{\AG\triv}(F[t])$) such that $\PP(n)=\QQ(n)$ for~$n\gg 0$. The result follows by the above reasoning and by the fact that $\QQ$ is saturated. 
\end{proof}

\subsection{Classification of semiprime saturated operads of
GK-dimension 1}
\label{zzsec5.2}

In this subsection we study saturated operads of GK-dimension 1
that are semiprime. The following result is a slight generalization
of Theorem \ref{zzthm0.4}.

\begin{theorem}
\label{zzthm-new3-5.4} 
Let $\PP$ be a finitely generated semiprime saturated operad of
GK-dimension 1. Then $\PP$ is isomorphic to a finite direct sum
of operads of the form $\Com^{\{w\}}_{F}$ or $\Ope^{\{2w\}}_{F}$,
where  $F$ is a finite field extension of $\Bbbk$ and  $w$ is a positive integer.
\end{theorem}

 Note that, in the direct sum in Theorem \ref{zzthm-new3-5.4}, the fields $F$ (resp. $w$) in different summands are not necessarily the same. We need a few easy lemmas. 

\begin{lemma}
\label{zzlem-new3-5.5}
A finite direct sum of operads of the form $\Com^{\{w\}}_{F}$ or $\Ope^{\{2w\}}_{F}$ is saturated  as an operad over~$\Bbbk$,
where $F$ is a finite field extension of $\Bbbk$ and $w$ is a positive integer.
\end{lemma}

\begin{proof}
 Note that if~${\text{char}}\; \Bbbk= 2$, then $\Ope^{\{2w\}}_{F}=\Com^{\{2w\}}_{F}$. We just prove for the case when~${\text{char}}\; \Bbbk\neq 2$ because the proof for the other case is similar.

Suppose ${\text{char}}\; \Bbbk\neq 2$. 
Let~$A_i=F_i[t_i]$ be a prime PGC algebra of odd or even type satisfying the following: (i) $\deg(t)=w_i$ if $A$ is of even type; (ii) $\deg(t)=2w_i$ if $A$ is of odd type. Then $\PP_i:=G_{\AG\triv}(A_i)$ is~$\Com^{\{w_i\}}_{F_i}$ or $\Ope^{\{2w_i\}}_{F_i}$ by Lemma~\ref{zzlem-new3-4.9}. 
Since each~$A_i$ is a finitely generated commutative $\Bbbk$-algebra and every left ideal of~$\PP_i$ corresponds to an ideal of $A_i$, we deduce that $\PP_i$ is left noetherian.
Therefore, 
$\PP_1\oplus \PP_2\oplus \cdots \oplus\PP_m$ is left noetherian for all~$m\geq 1$. By induction on $n$ and by Propositon~\ref{zzpro-new-3.4}, $\PP_1\oplus \PP_2\oplus \cdots \oplus\PP_n$ is saturated.
\end{proof}

Note that the operad $\Com_{\{w\}}$ with $w\geq3$ is prime with GK-dimension 1, but it is not saturated.

\begin{lemma}
\label{zzlem-new3-5.6}
Let $\PP$ be a finitely generated operad of GK-dimension
1. Let $I$ be a prime ideal of $\PP$ and $J$ be
an ideal strictly containing $I$. Then $\PP/J$ is
finite dimensional.  
\end{lemma}

\begin{proof}
Without loss of generality, we can assume that $I=0$ and
$\PP$ is prime.  By Theorem~\ref{zzthm-new-2.4}(2), $\PP$ is almost $\AG\triv$.
By the proof of  Theorem~\ref{zzthm5.1}, we have~$\PP=G_{\AG\triv}(A)$ if~${\text{char}}\; \Bbbk\neq 2$; and~$\PP=G_{\SG\triv}(A)$ if~${\text{char}}\; \Bbbk=2$,
 for some prime commutative graded algebra~$A$ of GK-dimension 1. In particular, $A$ is a domain. So every nonzero element in $A$ is regular.
 It follows that $A/I'$ is finite dimensional for every nonzero ideal $I'$.
This implies that $\PP/J$ is finite dimensional for every nonzero ideal $J$ of $\PP$ as desired.
\end{proof}

The following invariant is convenient for the proof of
Theorem \ref{zzthm-new3-5.4} and probably useful for studying
operads of higher GK-dimension. Suppose  $\PP$ has an integral
GK-dimension $h\geq 1$. The {\it multiplicity} of
$\PP$ is defined to be
\begin{equation}
\label{E5.6.1}\tag{E5.6.1}
{\mathfrak m}(\PP)
:=(h-1)! \limsup_{n\to\infty} \frac{\dim \PP(n)}{n^{h-1}}.
\end{equation}
In this paper we only consider the case when $h=1$ (and when operads are locally finite). 
Then 
$${\mathfrak m}(\PP)=\limsup_{n\to\infty} \dim \PP(n)$$
and it 
is finite if and only if $\dim \PP(n)$ is uniformly
bounded.  
For instance, let $F$ be a finite field extension of~$\Bbbk$ and let~$\PP=\Com_{F}$. Then we have ${\mathfrak m}(\PP)=\dim_{\Bbbk} F$.
If $\PP$ is of GK-dimension 0, by convention, we set~${\mathfrak m}(\PP)=0$.
 It is easy to construct two finitely generated operads
$\PP$ and $\QQ$ of GK-dimension 1 (or higher) such that
${\mathfrak m}(\PP\oplus \QQ)< {\mathfrak m}(\PP) +
{\mathfrak m}(\QQ)$.  For instance, let~$A=\Bbbk[x]$ with $\deg x=2$ and $\PP=G_{\SG\triv}(A)$. Then ${\mathfrak m}(\PP)=1$. Let $B$ be the commutative ring $\Bbbk[x]\oplus (\oplus_{i=1}^{3} a_i \Bbbk[x])$ with $\deg x=2$ and
$\deg a_i=1$ for $i=1,2,3$. The multiplication of $B$ is determined by $a_i a_j=0$ for all $i,j$. 
Let $\QQ=G_{\SG\triv}(B)$. Then ${\mathfrak m}(\QQ)=3$. It is easy to see that
$${\mathfrak m}(\PP\oplus \QQ)=3<4={\mathfrak m}(\PP)+{\mathfrak m}(\QQ).$$ 

\begin{lemma}
\label{zzlem-new3-5.7}
Let $\PP$, $\PP_i$, and $\QQ$ be operads of GK-dimension at most 1.
\begin{enumerate}
\item[(1)]
If $\PP$ is a suboperad of $\QQ$, then ${\mathfrak m}(\PP)
\leq {\mathfrak m}(\QQ)$.
\item[(2)]
If $\PP$ is a quotient operad of $\QQ$, then ${\mathfrak m}(\PP)
\leq {\mathfrak m}(\QQ)$.
\item[(3)]
If $\PP$ is almost isomorphic to $\QQ$, then ${\mathfrak m}(\PP)
={\mathfrak m}(\QQ)$.
\item[(4)] If
$\PP$ is finitely generated,
then ${\mathfrak m}(\PP)$ is a positive integer.
\item[(5)]  If
$\PP_i$ for $i=1,2,\cdots, n$ are finitely generated prime
operads, then
$${\mathfrak m}(\PP_1\oplus \cdots \oplus \PP_n)
=\sum_{i=1}^n {\mathfrak m}(\PP_i). $$
\item[(6)]  Every finitely generated operad of GK-dimension 1 maps surjectively to a prime operad of GK-dimension 1.  Moreover, if
$\PP_i$ for $i=1,2,\cdots, n$ are finitely generated operads of
GK-dimension 1, then
$${\mathfrak m}(\PP_1\oplus \cdots \oplus \PP_n)
\geq n.$$
\end{enumerate}
\end{lemma}

\begin{proof} (1-4) Easy.

(5) By part (3) and Corollary \ref{zzcor-new3-5.3}, we may assume
that each $\PP_i$ is of the form $\Com^{\{w\}}_{F}$ or
$\Ope^{\{2w\}}_{F}$ where $F$ is a finite field extension of $\Bbbk$. In this case, the assertion follows from a
direct computation.

(6) It suffices to show that every $\PP_i$
maps surjectively to a prime operad of GK-dimension 1, since then the second statement follows from parts~(2) and~(5).  Define
$$\Omega=\{I\unlhd \PP \mid \GKdim (\PP/I )=1\}.$$ Assume that~$\PP$ is generated by~$\{x_1, \dots, x_m\}$ such that~$\Ar(x_i)< M $ for all $1\leq i\leq m$.
Then clearly~$\{0\}\in \Omega$. Let~$I_1\subseteq I_2\subseteq \dots \subseteq I_n\subseteq \dots$ be an ascending chain of elements in~$\Omega$ and set~$I=\cup_{t\geq 1} I_t$.  We claim that~$I$ lies in $\Omega$.  Assume to the contrary that $I\not\in \Omega$ or equivalently $\GKdim (\PP/I )=0$. Then there exists an integer~$N>0$ such that~$\PP(n)=I(n)$ for all~$n\geq N$. Since $\PP$ is locally finite, there exists an integer $t>0$ such that for every~$f\in \PP(n)$ with~$N\leq n\leq N+M$, we have~$f\in I_t(n)$, and since $\PP$ is generated by~$\{x_1, \dots, x_m\}$ such that~$\Ar(x_i)< M $ for all $1\leq i\leq m$, ~$I_t(n)=\PP(n)$ for all~$n\geq N$, which contradicts the fact that~$I_t\in \Omega$. So we have~$I\in \Omega$.
By the Zorn's lemma, there exists a maximal element $I_{\max}$ in~$\Omega$.
Now we show that~$I_{\max}$ is prime. Assume~$I_{\max} \subsetneq J_1, J_2\unlhd \PP$ and~$J_1\circ J_2\subseteq I_{\max}$.
Since~$I_{\max}$ is a maximal element in~$\Omega$, we obtain~$\GKdim(P/J_1)=\GKdim(P/J_2)=0$, and thus there exists an integer~$N'>0$ such that~$\PP(n)=J_1(n)=J_2(n)$ for all~$n\geq N'$. Therefore, for every monomial~$\mu\in \PP(n)$ such that if ~$\Ar(\mu)$ is large enough, then there exist~$\mu_1\in J_1(n_1)$ and~$\mu_2\in J_2(n_2)$ satisfying~$n_1,n_2\geq N'$ and~$\mu=\mu_1\circ_q\mu_2\in I_{\max}$ for some integer~$q\leq n_1$. This implies that~$\GKdim(\PP/I_{\max})=0$, which yields a contradiction. So~$I_{\max}$ is prime and the proof is completed.
\end{proof}

\begin{lemma}
\label{zzlem-new3-5.8}
Let $\PP$ be a finitely generated operad of GK-dimension $1$.
Let $\{I\}$ be the set of distinct prime ideals
such that $\GKdim (\PP/I )=1$. Then
$1\leq |\{I\}|\leq {\mathfrak m}(\PP)$.
\end{lemma}

\begin{proof}
By Lemma \ref{zzlem-new3-5.7}(4), ${\mathfrak m}(\PP)$ is a positive
integer.  By Lemma~\ref{zzlem-new3-5.7}(6), we have $1\leq |\{I\}|$. Let
$\{I_1,\cdots,I_n\}$ be an arbitrary finite subset of $\{I\}$. Let
$K=\bigcap_{j=2}^n I_j$. Since  $I_2\circ (I_3\circ\cdots(I_{n-1}\circ I_n)\cdots)
	\subseteq K$ and each $I_j$ is prime, $K$ is not a subideal of
$I_1$.  Otherwise, there exists some integer~$j$ satisfying~$2\leq j\leq n$ and~$I_j\subsetneq I_1$. By Lemma \ref{zzlem-new3-5.6},  we obtain that~$\GK(\PP/I_1)=0$, yielding a contradiction.   By Lemma \ref{zzlem-new3-5.6} again, $\PP/(K+I_1)$ is finite dimensional.
Pick $w$ such that $(K+I_1)_{\geq w}=\PP_{\geq w}$. Consider the
natural morphism
$$f: \PP/(K\cap I_1)\to \PP/I_1\oplus \PP/K$$
which is injective by construction. Since $(K+I_1)_{\geq w}=
\PP_{\geq w}$, $f$ is an almost isomorphism. Similarly, or by
induction on $n$, there is an almost isomorphism $g:\PP/K\to
\oplus_{i=2}^n \PP/I_i$. Put these together we have an almost
isomorphism
\begin{equation}
\label{E5.8.1}\tag{E5.8.1}
f': \PP/(\cap_{i=1}^n I_i)=
\PP/(K\cap I_1)\to \oplus_{i=1}^n \PP/I_i.
\end{equation}
Then we obtain

$$\begin{aligned}
	n&\leq {\mathfrak m}(\oplus_{i=1}^n \PP/I_i) \quad \; \quad \; {\text{ by Lemma \ref{zzlem-new3-5.7}(5)(4)}}\\
	&={\mathfrak m}(\PP/(\cap_{i=1}^n I_i)) \qquad {\text{ by Lemma \ref{zzlem-new3-5.7}(3)}}\\
	&\leq {\mathfrak m}(\PP) \qquad \qquad\qquad {\text{ by Lemma \ref{zzlem-new3-5.7}(2)}}.
\end{aligned}
$$
Since $n$ is an arbitrary integer no more than  $|\{I\}|$,
the assertion follows.
\end{proof}

Now we are ready to prove Theorem \ref{zzthm-new3-5.4}.
\begin{proof}[Proof of Theorem \ref{zzthm-new3-5.4}]  We first show that, there exist finitely many prime ideals of~$\PP$ such that their intersection is 0.
By Lemma \ref{zzlem-new3-5.8},  there exist  finitely  many prime ideals,
say $\{I_1,\cdots,I_n\}$, such that each $\PP/I_i$ is prime of
GK-dimension 1. Let~$I=\bigcap_{i=1}^n I_i$.

Assume that~$J$ is a prime ideal of $\PP$ such that $\PP/J$ is finite dimensional. And let~$\Omega$ be the set consisting of all such prime ideals. Since $\PP/J$ is prime and finite dimensional, we deduce that $(\PP/J)_{\geq 2}=\{0\}$, so
 each $J$ contains $\PP_{\geq 2}$. Since~$\PP$ is semiprime, we have
 $$\PP_{\geq2}\cap I \subseteq \bigcap_{J\in \Omega} J \cap \bigcap_{i=1}^n I_i=\{0\}. $$
 This implies that $ \PP_{\geq2}\circ I =\{0\}$.  Since~$\PP$ is saturated, we have~$I\subseteq \tau^l(\PP)= \{0\}$.

Now we obtain an injective
morphism
$$f: \PP\to \oplus_{i=1}^n \PP/I_i.$$
By \eqref{E5.8.1}, one sees that $f$ is an almost
isomorphism. Let $\QQ_i$ be the saturation of~$\PP/I_i$. By Corollary \ref{zzcor-new3-5.3}, each $\QQ_i$ is of the form $\Com^{\{w\}}_{F}$ or $\Ope^{\{2w\}}_{F}$. Then $\oplus_{i=1}^n \QQ_i$ is the
saturation of $\PP$ by Lemma~\ref{zzlem-new3-5.5} and Proposition \ref{zzpro-new-3.4}(2).
Since $\PP$ is saturated, we obtain that
$$\PP\cong \oplus_{i=1}^n \PP/I_i
\cong \oplus_{i=1}^n \QQ_i.$$
 The assertion
follows.
\end{proof}

\subsection{Classification of saturated linear operads.}
\label{zzsec5.3}
In this subsection we consider finitely generated
saturated operads with Hilbert series $\frac{t}{1-t}$ and
prove Theorem \ref{zzthm0.7}. Note that Theorem \ref{zzthm0.7}
fails if $\PP$ is not finitely generated, see Example
\ref{zzex6.3}.

Now we are ready to prove Theorem \ref{zzthm0.7}. 

\begin{proof}[Proof of Theorem \ref{zzthm0.7}] Suppose ${\text{char}}\; \Bbbk\neq 2$.
Let $\PP$ be a finitely generated saturated linear operad.
Clearly, $\PP$ is a left torsionfree operad of GK-dimension 1.   By Theorem~\ref{zzthm-new-2.4}(2) and by \cite[Lemma 3.13(2)]{LQXZ25}, $\PP$ is  $\AG\triv$. So
by Theorem~\ref{zzthm-new-4.5}(1), we have $\PP=G_{\AG\triv}(A)$
for some saturated PGC algebra~$A$.  Moreover, $A$ is linear. Therefore, $A$ is a finitely generated saturated PGC algebra with Hilbert series $\frac{1}{1-t}$.

Assume that~$A$ is generated by a finite set~$X\subseteq A_o\cup A_e$.  Then as in the proof of Lemma~\ref{zzlem-new-3.10}, there exists a regular homogeneous element~$x\in X$. Assume that~$x$ is of degree~$m\geq 2$. Then it follows that~$A_{\geq m}=xA$ since the Hilbert series of~$A$ is $\frac{1}{1-t}$. Moreover, since $A_o$ and $A_e$ are ideals, it follows that either~$A_o$ or $A_e$ contains $A_{\geq m}$. 
Consequently, the other 
is finite dimensional. Therefore, by the construction of $G_{\AG\triv}(A)$, either $\PP_{\triv}$ or $\PP_{\sign}$ is finite dimensional. By Lemma \ref{zzlem-new-3.2}(3), either
$\PP_{\sign} \subseteq \tau^l(\PP)=\{0\}$ or $\PP_{\triv} \subseteq \tau^l(\PP)=\{0\}$. So~$\PP$ is $\SG\triv$ or $\SG\sign$, and thus~$A$ is of even type or odd type.

 Finally,  by Lemma~\ref{zzlem-new-3.10}, $A$ is
isomorphic to $B\{c\}$ given in Example \ref{zzex-new-3.6} (resp. Example \ref{zzrema-new-3.7}), where~$B$ is nilpotent. 
  Since~ $\PP=G_{\AG\triv}(A)$,  it follows that $\PP$  is of the form $\Lin_e(B)$ or $\Lin_{o}(B)$. 
\end{proof}

We conclude this section with an easy observation on~$G_{\AG\triv}(A)$ when $A=B\{c\}$ for some specific $B$. If~$B=\Bbbk$, $\deg(c)=1$ and $A$ is a PGC algebra of even type, then~$G_{\AG\triv}(A)=\Com$. If~$B=\Bbbk\oplus \Bbbk x_1$, $x_1^2=0$, $\deg(c)=2$ and $A$ is a PGC algebra of odd type, then~$G_{\AG\triv}(A)=\Mas$.
\section{Comments, remarks and examples}
\label{zzsec6}

When $\overline{\Bbbk}$ is infinite dimensional over
$\Bbbk$, there are infinitely generated prime operads
of GK-dimension 1, see the next example.

\begin{example}
\label{zzex6.1}
By calculus, the function
$$f(x):=(x+1)\ln (x+1)- x\ln(x),\quad x\geq 1$$
is positive and increasing to infinity when $x\to \infty$.
For each positive integer $m$, let~$\lambda(m)$ be the smallest
positive integer such that
$f(\lambda(m))\geq m$.

Suppose that $\overline{\Bbbk}$ is infinite dimensional over
$\Bbbk$. Then there is a strictly increasing chain of finite
field extensions $F_s/\Bbbk$ (with $\dim  F_1>1$) such
that $\lim_{s\to \infty} \dim  F_s=\infty$.
We now construct a connected graded algebra
$A=\oplus_{i=0}^{\infty} A_i$ that is a
$\Bbbk$-sub-algebra of $\overline{\Bbbk}[t]$ defined by
$$A_i=\begin{cases}
\Bbbk t^i & \quad {\text{if}} \;\; 0\leq i <\lambda(\dim  F_1),\\
 F_{s} t^i & \quad
{\text{if}} \;\; \lambda(\dim  F_s) \leq i <\lambda(\dim  F_{s+1}).
\end{cases}$$
By definition, for all~$i$ such that~$\lambda(\dim  F_s) \leq i <\lambda(\dim  F_{s+1})$, we have
$$\dim  A_i=\dim  F_s
\leq f(\lambda(\dim  F_s))\leq f(i).$$
Therefore, there is a constant $a$ such that for all $n\geq 1$,
$$\begin{aligned}
\sum_{i=0}^{n} \dim  A_i &\leq\lambda(\dim  F_1)+
\sum_{i=\lambda(\dim  F_1)}^n f(i) \\
&\leq a+ \sum_{i=1}^{n} f(i)=a+(n+1)\ln(n+1),
\end{aligned}
$$
which implies that $\GKdim (A)=1$. It is clear that
$A$ is not finitely generated over~$\Bbbk$.
Let $\PP=G_{\SG\triv}(A)$. Then $\PP$ is a prime operad
of GK-dimension 1 that is not finitely generated. It is
clear that $\dim \PP(n)$ is not uniformly bounded. Since
$A$ is commutative, it follows that every element
in $\PP$ is central.
\end{example}

Next we construct a prime operad of GK-dimension 1
that is not almost $\AG\triv$.

\begin{example}
\label{zzex6.2}
We start with constructing an infinitely generated connected
graded algebra $A$ generated by $\{x_s\}_{s=1}^{\infty}$ of degrees
$d_s:=\deg x_s\geq 1$ for all $s\geq 1$. The sequence
$\{d_s\}_{s\geq 1}$ will be strictly increasing and chosen
carefully as in \eqref{E6.2.2} below. Let
$A=\Bbbk\langle x_s\mid s\geq 1\rangle/R$, where the relation
ideal $R$ is generated by the elements of the form
\begin{equation}
\label{E6.2.1}\tag{E6.2.1}
x_s x_{i_1}\cdots x_{i_w} x_s
\end{equation}
such that $w\geq 0$ and $i_{j}<s$ for all $1\leq j\leq w$. For
example, $x_1^2, x_2^2, x_2 x_1 x_2, x_2x_1^2x_2$ are the first
few elements in the relation ideal $R$. It is easy to check that
$A$ is prime and infinitely generated and the center of $A$
is trivial. Let $J_i$ be the ideal of $A$ generated by $\{x_1,\cdots,x_i\}$.
Then $A/J_i$ is another prime and infinitely generated connected
graded algebra with trivial center.

Let ${\mathcal S}_A$ be the symmetric operad provided by
\cite[Construction 8.1]{QXZZ20}. Then~${\mathcal S}_A$ is a
prime and infinitely generated connected operad with no
central element. Now we want to estimate the GK-dimension of
${\mathcal S}_A$.

By \cite[Construction 8.1]{QXZZ20}, $\dim {\mathcal S}_A(i)
=i \dim A_{i-1}$. Since there are infinitely many $i$
with nonzero $\dim A_{i-1}$, we have that

\begin{enumerate}
\item[(1)]
$\GKdim ({\mathcal S}_A)\geq 1$.
\end{enumerate}

Next we will choose the sequence $\{d_s:=\deg x_s\}_{s\geq 1}$
such that $\GKdim ({\mathcal S}_A)=1$.
Fix any $d_1\geq 1$ (which could be $1$), we define $d_{s+1}$
based on $\{d_1,\cdots,d_s\}$. Let~$A^{\langle s\rangle}$ be the
subalgebra of $A$ generated by $\{x_1,\cdots,x_s\}$. By the
relations of the form~\eqref{E6.2.1}, $A^{\langle s\rangle}$ is
finite dimensional. Let $\alpha_s$ be  $\dim A^{\langle s \rangle}$
and $\beta_s$ be the highest degree of a nonzero element in
$A^{\langle s\rangle}$. We define
\begin{equation}
\label{E6.2.2}\tag{E6.2.2}
d_{s+1}:=1+\lfloor e^{2 \alpha_s^2(1+2\beta_s)^2}\rfloor.
\end{equation}
We will verify that $\GKdim ({\mathcal S}_A)=1$. By definition,
$A_i=0$ when~$\beta_s<i < d_{s+1}$. For~$d_{s+1}\leq i\leq \beta_{s+1}$, we have the following estimate
\begin{equation}
\label{E6.2.3}\tag{E6.2.3}
\dim A_i\leq \alpha_s^2,
\end{equation}
which follows from the fact that every element in $A_i$
with $d_{s+1}\leq i\leq \beta_{s+1}$ is a linear combination
of elements of the form $fx_{s+1} g$ with~$f,g\in A^{\langle s\rangle}$. It is clear that
\begin{equation}
\label{E6.2.4}\tag{E6.2.4}
\beta_{s+1}=d_{s+1}+2 \beta_s.
\end{equation}
We now show that, there is a positive integer $C>1$ such
that
\begin{equation}
\label{E6.2.5}\tag{E6.2.5}
\sum_{i=1}^n \dim {\mathcal S}_A(i)\leq C n\ln n
\end{equation}
for all $n\geq d_1+1$.

For any $n\geq d_1+1$, there are two possibilities: one is
$d_{s}+1\leq n \leq \beta_{s}+1$ for some~$s$ and the
other is $\beta_{s}+1<n<d_{s+1}+1$ for some $s$. We consider
these two cases:

Case 1: $d_{s}+1\leq n \leq \beta_{s}+1$. We prove
\eqref{E6.2.5} by induction on $s$. When $s=1$, this
follows by taking some large positive number $C$.   Assume
that \eqref{E6.2.5} holds when $d_{s}+1\leq n \leq \beta_{s}+1$.
Now consider $n$ with $d_{s+1}+1\leq n \leq \beta_{s+1}+1$, we obtain
$$\begin{aligned}
\sum_{i=1}^n \dim {\mathcal S}_A(i)
&=\sum_{i=1}^{\beta_{s}+1}
\dim {\mathcal S}_A(i)
+\sum_{i=d_{s+1}+1}^{n}
\dim {\mathcal S}_A(i) \\
&\leq C (\beta_{s}+1) \ln (\beta_{s}+1)
+\sum_{i=d_{s+1}+1}^{\beta_{s+1}+1}
\dim {\mathcal S}_A(i) \quad\quad {\text{by induction hypo.}} \\
&\leq C (\beta_{s}+1) \ln (\beta_{s}+1)
+\sum_{i=d_{s+1}+1}^{d_{s+1}+2\beta_s+1}
\dim {\mathcal S}_A(i) \quad\qquad\qquad {\text{by \eqref{E6.2.4}}} \\
&\leq C (\beta_{s}+1) \ln (\beta_{s}+1)
+\sum_{i=d_{s+1}+1}^{d_{s+1}+2\beta_s+1}
(d_{s+1}+2\beta_s+1) \alpha_s^2 \quad {\text{by \eqref{E6.2.3}}}\\
&\leq C (\beta_{s}+1) \ln (\beta_{s}+1)
+(1+2\beta_s) (d_{s+1}+2\beta_s+1) \alpha_s^2 \\
&\leq C (\beta_{s}+1) \ln (\beta_{s}+1)
+(1+2\beta_s)^2 d_{s+1} \alpha_s^2 \\
&\leq C (\beta_{s}+1) \ln (\beta_{s}+1)
+\frac{1}{2} d_{s+1} \ln d_{s+1}  \quad\qquad\qquad\qquad\quad {\text{by \eqref{E6.2.2}}} \\
&\leq  \frac{1}{2} C d_{s+1}\ln d_{s+1}
+\frac{1}{2} d_{s+1} \ln d_{s+1}
\leq C d_{s+1}\ln d_{s+1}\\
&\leq C n \ln n.
\end{aligned}
$$

Case 2: 
$\beta_{s}+1<n<d_{s+1}+1$ for some $s$. 
Since $A_j=0$
for $\beta_s<j<d_{s+1}$, we have ${\mathcal S}_A(i)=0$  for $\beta_s+1<i<d_{s+1}+1$. 
 By induction hypothesis and by \eqref{E6.2.5}, we have 
$$\begin{aligned}
\sum_{i=1}^n \dim {\mathcal S}_A(i)
&=\sum_{i=1}^{\beta_{s}+1}
\dim {\mathcal S}_A(i)
\leq C (\beta_s+1) \ln (\beta_s+1)
\leq C n \ln n
\end{aligned}
$$
as desired.  Note that
\eqref{E6.2.5} implies that~$\GKdim ({\mathcal S}_A)\leq 1$.
Combining with part~{(1)}, we have

\begin{enumerate}
\item[(2)]
When we choose $d_{s+1}$ as in \eqref{E6.2.2},
$\GKdim ({\mathcal S}_A)=1$.
\end{enumerate}

The following assertions  can be easily proved by verifying
the corresponding assertions for the algebra $A$ and thus the
details are omitted.

\begin{enumerate}
\item[(3)]
There is an ascending chain of ideals $\{I_i\}_{i\geq 1}$
of ${\mathcal S}_A$ such that ${\mathcal S}_A/I_i$ is again
a prime operad of GK-dimension 1. This follows from the fact that
${\mathcal S}_A/I_i\cong {\mathcal S}_{A/J_i}$ where
$J_i$ is the ideal of $A$ described at the beginning
of this example.
\end{enumerate}

\begin{enumerate}
\item[(4)]
Elements of arity $\geq 2$ in ${\mathcal S}_A$ is
not central.
\item[(5)]
${\mathcal S}_A$ is a union of finite dimensional
suboperads.
\item[(6)]
Every finitely generated suboperad of ${\mathcal S}_A$
is finite dimensional.
\end{enumerate}

For the rest of this example, we suppose that $\Bbbk$ is not
algebraically closed. Let~$F$ be a non-trivial finite field
extension of $\Bbbk$. Let $\QQ$ be the ``base-change'' operad of 
${\mathcal S}_{A}$ defined to be ${\mathcal S}_{A}\otimes F$.  In general, let $\PP$ be an operad over $\Bbbk$ and $R$ be a commutative
$\Bbbk$-algebra. We define a ``base-change'' operad $\PP_R$
over $\Bbbk$ as follows:
\begin{enumerate}
\item[(i)]
$\PP_R(n):=\PP(n)\otimes R$ for all $n\geq 0$, where $\otimes$
means $\otimes_{\Bbbk}$,
\item[(ii)]
for $x\in \PP(n)$ and $r\in R$, $(x\otimes r)\ast \sigma:=
(x\ast \sigma) \otimes r$ for all $\sigma\in \SG_n$,
\item[(iii)]
for $x\in \PP(n)$, $y\in\PP(m)$, $r\in R$, and $s\in R$,
$(x\otimes r)\circ_i (y\otimes s):=(x\circ_i y)\otimes (rs)$
for all $1\leq i \leq n$,
\item[(iv)]
the identity of $\PP_R$ is $1_{\PP}\otimes 1_{R}$.
\end{enumerate}
It is easy to see that $\PP_R$ is an operad over $\Bbbk$.

Note that $\QQ$ is also isomorphic to ${\mathcal S}_{A\otimes F}$, where $(A\otimes F)_{n}$ is defined to be $A_n\otimes F$. 

\begin{enumerate}
\item[(7)]
Similar to ${\mathcal S}_A$, $\QQ$ is prime. Let
$f\in F \setminus \Bbbk$. Then $f$ is a central element of arity 1.
Every nonzero element $x\in \QQ$ of arity $\geq 2$ is not central.
\end{enumerate}
\end{example}

The next example is an infinitely generated linear operad.

\begin{example}
\label{zzex6.3}
Let $A$ be the connected commutative graded algebra generated by
$\{x_0,x_1,\cdots\}$ with $\deg x_i=2^i$ for all $i\geq 0$ and
subject to the relations
$$x_0^2=x_1^2=x_2^2=\cdots =x_i^2 = \cdots =0$$
for all $i$. It is easy to check that $A$ is torsionfree with Hilbert series~$\frac{1}{1-t}$. Moreover, we have the following cancellation property: for all~$f,f'\in A$ of degree less than $2^{l}$, the equation $x_lf=x_lf'$ forces~$f=f'$. 
 Let $\PP=G_{\SG\triv}(A)$ be as constructed in
\cite[Lemma 3.7]{LQXZ25}.
Then~$\PP$ is a left (and right) torsionfree operad
with Hilbert series $\frac{t}{1-t}$ and~$\PP$
 is infinitely generated and of GK-dimension~1. Let $I_i$ be the ideal of
$\PP$ generated by~$\{x_0,\cdots,x_i\}$. Then it is easy to check that
$\PP/I_i$ is also left (and right) torsionfree and of
GK-dimension~1.   Now we show that~$\PP$ and $\PP/I_i$ $(i
\geq 0)$ are saturated. 
Define~$A^{(i)}:=A/J_i$, where~$J_i$ is the ideal of~$A$ generated by $\{x_0,\cdots,x_i\}$.  By Theorem~\ref{zzthm-new-4.2}, it suffices to show that~$A$ and $A^{(i)}$ $(i\geq 0)$ are 
saturated.
Let~$C$ be $A$ or $A^{(i)}$.   
Then the elements $\{\alpha_t:= x_{i+t}\}_{t\geq 1}\subseteq C $ satisfy  
Condition ($\ast$) in Lemma \ref{saturated-condi}, and thus the assertion
follows immediately.
\end{example}

\begin{definition}
\label{zzdef-new-6.4}
Let $s$ be a positive integer.
An operad $\PP$ is called {\it $s$-linear} if
\begin{enumerate}
\item[(i)]
$\PP$ is saturated,
\item[(ii)]
$H_{\PP}(t)=\frac{t}{(1-t)^s}$.
\end{enumerate}
\end{definition}

Note that Theorem \ref{zzthm0.7} gives a classification
of $1$-linear operads.

\begin{example}
\label{xxex6.6}
Here are some examples of $2$-linear operads.
\begin{enumerate}
\item[(1)]
Let $A$ be a connected  saturated commutative graded algebra
with Hilbert series $\frac{1}{(1-t)^2}$.
Then $G_{\SG\triv}(A)$ is $2$-linear.
\item[(2)]
Let $A$ be a connected saturated graded commutative algebra (considered
as a PGC algebra of odd type). Suppose the Hilbert series of
$A$ is $\frac{1}{(1-t)^2}$. Then $G_{\AG\triv}(A)$ is $2$-linear.
One special case is $\Mas^2_2$ given in \cite[Example 6.2]{LQXZ25}.
\item[(3)]
Let $\Bbbk[t]$ be the polynomial ring with $\deg t=1$. Let
$\PP:={\mathcal S}_{\Bbbk[t]}$ be the operad constructed in
\cite[Construction 8.1]{QXZZ20}. Then $\PP$ is $2$-linear.
\end{enumerate}
\end{example}

The next question is a sub-question of Question~\ref{zzque0.8}(1).

\begin{question}
\label{zzque-new-6.7}
How can we classify all finitely generated $2$-linear operads?
\end{question}

\section*{Acknowledgments}
The  authors thank Yanhua Wang for her valuable suggestions. Y. Li and X.-G. Zhao would like to thank C.-H. Li and the department of
Mathematics in Southern University of Science and Technology for the hospitality during their visit. Y. Li was
partially supported by the National Science Foundation of
China (No. 11501237).
Z.-H. Qi was partially supported by the National Science
Foundation of China (No. 11771085). Y.-J. Xu was partially
supported by Shandong Provincial Natural Science Foundation (No. ZR2024MA052) and National Science Foundation of China
(No. 11871301). J.J. Zhang was partially
supported by the US National Science Foundation (Nos. DMS-2001015 and DMS-2302087). Z.-R. Zhang was supported by the
Guangdong Basic and Applied Basic Research Foundation (No. 2024A1515013122).
X.-G. Zhao was partially supported by
the Guangdong Basic and Applied Basic Research Foundation (No. 2023A1515011690) and the Characteristic Innovation Project of Guangdong Provincial Department of Education (Nos. 2023KTSCX145, 2025KTSCX145).

\section*{Conflict of Interest}
The authors declare that they have no conflict of interest.

\bibliographystyle{amsalpha}

\end{document}